%% file: main_arxiv_July_2026.tex
\documentclass[11pt]{amsart}

   \usepackage[a4paper, margin=100pt]{geometry}    
    \usepackage[T1]{fontenc}
    \usepackage[british]{babel}
    \usepackage[utf8]{inputenc}

    \usepackage{mathabx}

\normalsize 
    \addtocontents{toc}{\setcounter{tocdepth}{1}}

    \usepackage{xcolor}
    \usepackage{hyperref}
         \hypersetup{colorlinks, citecolor=blue, filecolor=blue, linkcolor=black, urlcolor=blue}
    \usepackage{soul}
    \usepackage[numbers]{natbib}    
   
    \usepackage{caption}
    \usepackage{subcaption}

    \usepackage{tikz-cd}

\input{macros.tex}

    \usepackage{bigints}
    \usepackage{xfrac}
    \usepackage{mathrsfs}
    \usepackage{comment}
    \usepackage{verbatim}
    \usepackage{adjustbox}
    \usepackage{graphicx}
    \usepackage{import}

    \newcommand{\dist}{\mathrm{dist}}
    \newcommand{\uv}{\underline{v}}

\begin{document}

    \title[cocycles with quasi-conformality]{Cocycles with Quasi-Conformality II:\\ Ergodic Measures with Positive Entropy}

\author[M. Nassiri]{Meysam Nassiri}
\address{School of Mathematics, Institute for Research in Fundamental Sciences (IPM), P.O. Box 19395-5746, Tehran, Iran}
\email{nassiri@ipm.ir}
\author[H. Rajabzadeh]{Hesam Rajabzadeh}
\address{School of Mathematics, Institute for Research in Fundamental Sciences (IPM), P.O. Box 19395-5746, Tehran, Iran}
\email{rajabzadeh@ipm.ir}
\author[Z. Reshadat]{Zahra Reshadat}
\address{School of Mathematics, Institute for Research in Fundamental Sciences (IPM), P.O. Box 19395-5746, Tehran, Iran}
\email{zahrareshadat@ipm.ir}

\date{}

\begin{abstract}As the second part of a series on linear cocycles over chaotic systems, this paper establishes a "multiple covering principle" that robustly yields positive-entropy ergodic measures supported on fiberwise uniformly bounded orbits. Using this mechanism, we prove that any continuous $\mathrm{SL}(d,\mathbb{R})$ cocycle over a positive-entropy subshift of finite type either admits a dominated splitting or can be $\mathcal{C}^0$-approximated by one that 
\emph{stably} supports such measures. 
The stability holds in Hölder topology which is optimal.
Additionally, for non-isometric cocycles, we show that the topological entropy of these bounded orbits is strictly less than that of the base subshift.
\end{abstract}

\maketitle
\tableofcontents

\section{Introduction}\label{sec:intro}
Let $f \colon X \to X$ be a homeomorphism on a metric space $X$ and $G$ be a closed subgroup of $\mathrm{GL}(d, \R)$. A continuous map $A \colon X \to G$ defines a linear cocycle over $f$, denoted by $(f, A)$, which acts on the trivial vector bundle $X \times \R^d$ via the skew-product dynamics:
$$F(x, v) = (f(x), A(x)v).$$

Linear cocycles are ubiquitous in dynamical systems and mathematics at large. They arise naturally as the derivative map of smooth dynamical systems on manifolds, in the study of random matrix products, in the spectral theory of Schrödinger operators, and in many other contexts. The central theme in the study of linear cocycles is understanding the asymptotic behavior of the sequence of fiber maps $A_n(x) := A(f^{n-1}x) \cdots A(f(x))A(x)$ as $n$ grows.

One of the robust asymptotic behaviors occurs when a cocycle admits a \emph{dominated splitting}. This is a projective form of hyperbolicity characterized by an invariant splitting of the fibers into subbundles with uniformly separated growth rates. Precisely, a dominated splitting for the cocycle $(f,A)$ is a continuous decomposition of the fibers over $X$ into $A$-invariant subbundles $\R^d=E_1(x)\oplus E_2(x)$ such that there exist constants $C > 0$ and $0 < \lambda < 1$ satisfying
$$ {\|A_n(x)|_{E_1(x)}\| . \|(A_n(x)|_{E_2(x)})^{-1}\| \leq C \lambda^n} $$
for all $x \in X$ and $n \geq 1$. It is well known that this property is stable under small continuous perturbations and dictates well-structured dynamical properties. These include a uniform exponential separation between two consecutive singular values of the cocycle under iteration, or the non-degeneracy of the Lyapunov spectrum with respect to all invariant measures on the base. It is known due to \cite{BochiGourmelon} that the former property is equivalent to the existence of domination. For a comprehensive discussion on dominated splitting and its dynamical consequences, we refer the reader to \cite{BDV,PujalsSambarino2,sambarino2016short,bochi2019anosov} and the references therein.
Given the structure imposed by a dominated splitting, a fundamental question naturally arises: {\it what are the typical behaviors of a cocycle in the absence of a dominated splitting?} The key to classifying typical cocycles without domination lies in understanding the existence and stability of non-exponential growth on the fibers. Foundational work by Furstenberg \cite{Furstenberg} on random matrix products,  extended to linear and smooth cocycles by Ledrappier \cite{Led86} and Avila and Viana \cite{Viana08, Avila-Viana}, indicates that typical $\cC^\alpha$ cocycles for every $\alpha>0$ exhibit non-zero Lyapunov exponents. Consequently, one expects exponential growth of norms for typical points, which intuitively suggests that orbits with sub-exponential or bounded growth are scarce and easily destroyed under small perturbations.

In the first paper of this series \cite{NRR1}, we demonstrated the $\cC^\alpha$-stable existence of uniformly bounded orbits for typical cocycles when domination fails. Specifically, it was shown that for volume-preserving cocycles over the shift map, any cocycle without domination can be $\cC^0$-approximated by cocycles where the set of points with \emph{fiberwise bounded orbit}  (i.e.,  points $x \in X$ for which the norms of $A_n(x)^{\pm 1}$ along their full orbit remain uniformly bounded) is $\cC^\alpha$-stably non-empty. In the present paper, we elevate this phenomenon from the existence of individual orbits to a macroscopic, dimensional scale. More concretely, by utilizing the mechanism developed in this paper, we establish that for all  those $\cC^\alpha$-open sets of cocycles, the set of bounded orbits actually contains a closed subsystem carrying positive topological entropy. 
More precisely, let $\mathrm{B}^{\pm}(A, \kappa)$ be the set of points $x\in X$ such that $\|A_n(x)^{\pm 1}\|\leq \kappa$ for all $n\in \Z$, and let ${\bfB}(A, \kappa)$ denote its maximal invariant subset\footnote{Invariant sets are assumed to be both forward and backward invariant unless stated otherwise.}. We establish the following theorem:

\begin{theoremain}\label{thm:A}
Let $d>1$ be an integer and $\alpha>0$. Suppose $\sigma:\Sigma_H\to \Sigma_H$ is a transitive subshift of finite type with positive topological entropy, and let $(\sigma,A)$ be a continuous $\sldr$ cocycle. If $(\sigma, A)$ admits no dominated splitting, then it can be $\cC^0$-approximated by $\cC^\alpha$-open sets $\cB$ of $\sldr$ cocycles with the following property: for each open set $\cB$, there exist constants $0<\gamma \leq \gamma' < h_{\mathrm{top}}(\sigma\big|_{\Sigma_H})$ and $\kappa>0$ depending on $\cB$ such that every cocycle $(\sigma,\tilde{A}) \in \cB$ satisfies
$$\gamma \leq h_{\mathrm{top}}(\sigma\big|_{{\bfB}(\tilde{A},\kappa)})\leq \gamma'.$$
\end{theoremain}

 \medskip

We proceed with a few remarks regarding Theorem \ref{thm:A}. First, in view of the well-known relationship between topological entropy and Hausdorff dimension of closed invariant subsets of the shift space, this result implies that the set of points with fiberwise bounded full orbit has strictly positive Hausdorff dimension. See Remark \ref{rmk:Hausdorff-dim} for further details.

Furthermore, points with fiberwise bounded orbits can be viewed as a generalization of periodic points that exhibit elliptic behavior along the fiber upon return. From this perspective, Theorem \ref{thm:A} continues the line of inquiry established in \cite{NRR1}, providing a refined characterization of elliptic versus hyperbolic behavior for cocycles. Additionally, because the set ${\bfB}(A,\kappa)$ is compact and invariant under the dynamics, combining Theorem \ref{thm:A} with the variational principle for entropy establishes the $\cC^{\alpha}$-stable existence of ergodic invariant measures with a {\it vanishing Lyapunov spectrum}. 
This advances the existing line of research concerning the robust existence of non-hyperbolic measures \cite{BochiBonattiDiaz2014, bochi2016robust, diaz2019entropy, bonatti2021robust, diaz2022variational, Martha2023}. Whereas the existing works often have restrictions on the dimension of the fiber, may provide measures of zero entropy (such as periodic measures), or guarantee only a single zero Lyapunov exponent,  our approach provides a unified framework for fibers of any dimension, ensuring the existence of positive-entropy invariant measures with a completely vanishing Lyapunov spectrum. This work also contributes to other lines of research, including \cite{Yoccoz2004, Avila_Bochi_Yoccoz, BochiRams_JMD} for linear cocycles generated by finite families of $2\times 2$ matrices, and \cite{mane1982ergodic, Bochi02, Bochi-Viana, ACW} on the genericity of zero exponents in the absence of hyperbolicity.

\medskip

While Theorem \ref{thm:A} establishes both an upper and a lower bound for the entropy of the set of bounded orbits, the arguments and tools required for each bound must be addressed separately. The lower bound relies on a new general phenomenon unveiled in this paper, which we encapsulate in the ``\emph{Multiple Covering Principle}''. We believe this principle is of independent interest; see, for instance, the recent work \cite{NZ1} regarding the role of multiple covering conditions in the stable intersection of Cantor sets. Indeed, the transition from the mere stable existence of bounded orbits \cite{NRR1} to the robust dimensional estimates presented in Theorem \ref{thm:A} fundamentally relies on this principle.

\begin{theoremain}[Multiple Covering Principle]\label{thm:B}
Let $(\sigma, A)$ be a locally constant linear cocycle. If $(\sigma, A)$ satisfies the covering condition, then it satisfies the multiple covering condition.
\end{theoremain}

To provide some context, the \emph{covering condition} for Iterated Function System (IFS) generated by a finite set $\cG\subset \gldr$ (which is equivalent to a locally constant cocycle over a full shift depending on the zeroth position) simply means there exists a non-empty open set $\cU \subset \gldr$ with compact closure such that $\overline{\cU} \subset \cG^{-1}\cU$. This is a highly robust property that allows for the construction of at least one trajectory for which all fiberwise iterations remain in $\cU$. For the precise definitions of the covering and multiple covering conditions in the general case, see Section \ref{sec:multiple-covering} (specifically, Definitions \ref{def:mutiple-covering} and \ref{def:gen-cov}). This \emph{multiple covering condition} which, roughly speaking, indicates the existence of several layers of covering, forces the existence of an exponentially growing number of independent itineraries, all of which exhibit uniformly bounded trajectories along fibers. Consequently, the set of base points with uniformly bounded fiberwise full orbits must have positive Hausdorff dimension, and the dynamics restricted to this set must carry positive entropy.

The following theorem states the core result for the lower bound, asserting that the multiple covering condition immediately implies positive topological entropy for the set of bounded orbits, given a certain bound.

\begin{theoremain}\label{thm:C}
Let $d\in \N$ and $\alpha>0$. Suppose $(\sigma,A)$ is a locally constant $G$-cocycle over a transitive subshift of finite type $\sigma:\Sigma_H\to \Sigma_H$ with positive topological entropy, where $G\subset \mathrm{GL}(d, \R)$ is a closed subgroup.
If $(\sigma,A)$ satisfies the multiple covering condition of $m$ layers for some $m>1$, then there exist $\varepsilon,\kappa,\gamma>0$ such that for every $\alpha$-H\"older cocycle $(\sigma,\tilde{A})$ with $d_{\cC^{\alpha}}(A,\tilde{A})<\varepsilon$, the set ${\bfB}(\tilde{A},\kappa)$ is non-empty, and moreover,
$$ h_{\mathrm{top}}(\sigma|_{{\bfB}(\tilde{A},\kappa)})\geq \gamma. $$
\end{theoremain}

 \medskip
Translating this covering property into the stability of linear cocycles requires careful estimates that strongly depend on the H\"older regularity of the underlying cocycles. By employing delicate H\"older estimates similar to those in \cite{NRR1}, we establish that the behavior generated by multiple covering is indeed stable in the $\cC^\alpha$ topology for $\alpha > 0$. The necessity of this H\"older regularity is strict: a recent result by Bochi \cite{Bochi2026} demonstrates that one can \emph{remove all quasi-conformal orbits via $\cC^0$ perturbations}, proving that such stability fails in the $\cC^0$ topology. 

We also emphasize that Theorem \ref{thm:C} holds for every closed subgroup of $\mathrm{GL}(d, \R)$, not only for $\mathrm{SL}(d, \R)$ as in Theorem \ref{thm:A}. Indeed, for volume-preserving cocycles, bounded orbits are equivalent to quasi-conformal orbits, meaning the condition numbers of the products $A_n$
 along a full orbit remain bounded. However, these concepts differ in general. While every point with a fiberwise bounded orbit clearly has a quasi-conformal orbit, the converse is generally false unless the cocycle takes values in a subgroup of volume-preserving matrices. For a general cocycle, one can associate a volume-preserving variant by factoring out an appropriate power of the determinant. With this construction, any quasi-conformal orbit of the initial cocycle corresponds to a fiberwise bounded orbit of its volume-preserving associate, ensuring that all our theorems have a counterpart for quasi-conformal orbits.

\medskip
Given that the set of fiberwise bounded orbits can support measures with positive entropy, it is natural to ask: \emph{can the measures supported on bounded orbits achieve the maximal entropy of the SFT?} The upper bound in Theorem \ref{thm:A} addresses this question. It is a consequence of the following rigidity theorem, which demonstrates that this extreme scenario cannot occur unless the cocycle itself is globally rigid, that is, the action of the cocycle is isometric with respect to a continuously varying inner product on the fibers (see Section \ref{sec:rigidity} for equivalent definitions).
{
\begin{theoremain}\label{thm:upperbound}
    Let $d>1$ be an integer and $\alpha>0$. Suppose $(\sigma,A)$ is an $\alpha$-H\"older continuous $\gldr$ cocycle over a transitive subshift of finite type $\sigma:\Sigma_H\to \Sigma_H$ with positive topological entropy. Assume that $(\sigma,A)$ is not continuously conjugate to an isometric cocycle.  Then for each $\kappa>0$ there exist a $\cC^0$ neighborhood $\cU$ of $(\sigma,A)$ and a constant $\eta > 0$
such that every $\tilde{A} \in\cU$,
\[h_{\rm top}(\sigma\big|_{{\bfB}(\tilde{A},\kappa)})\leq h_{\rm top}(\sigma\big|_{\Sigma_H})-\eta.\]
\end{theoremain}
} 
\medskip

{We make a few remarks regarding this result. First, the same statement holds for $\sldr$ cocycles. Due to the equivalence between fiberwise bounded orbits and quasi-conformal orbits for $\sldr$ cocycles, this immediately implies a gap for the entropy of the shift restricted to the set of quasi-conformal orbits, when the cocycle is not continuously conjugate to an isometric one. Furthermore, the assertion for $\sldr$ cocycles is equivalent to the analogous statement for $\gldr$ cocycles and quasi-conformal orbits. Specifically, if a $\gldr$ cocycle over a base dynamics as in Theorem \ref{thm:upperbound} does not preserve a continuously varying family of conformal structures on the fibers, then the entropy of the restriction to the set of quasi-conformal orbits stably exhibits a uniform gap (in the sense of the theorem) to the full entropy of the system.}

The upper bound in Theorem \ref{thm:A} follows from Theorem \ref{thm:upperbound}, combined with the genericity of non-vanishing Lyapunov exponents for volume-preserving H\"older cocycles \cite{Viana08},  and the rigidity achieved when the cocycle is bounded over all orbits \cite{Kalinin_JMD,KalininSadovskaya2018}. Note that \emph{the H\"older regularity in this theorem is necessary}: if the cocycle is only assumed to be continuous, one can construct examples for which every orbit is fiberwise bounded, yet the cocycle is not conjugate to an isometric one (see Example \ref{rmk:nonrigid-full-entropy}). While global boundedness implies rigidity for continuous cocycles over minimal systems \cite{BochiNavas,CoronelNavasPonce,reshadat_Thesis}, this is not true for cocycles over systems such as SFTs of positive entropy. This distinction is closely related to Liv\v{s}ic-type results and Gottschalk-Hedlund theorems, which deduce rigidity from bounded fiberwise data \cite{Kalinin2011,Quas}. 

It is also interesting to compare these results with large deviation estimates for the norm of random products of matrices, which correspond to the IFS generated by finite families of matrices. When the top Lyapunov exponent of the associated locally constant cocycle is positive with respect to a Bernoulli measure, these estimates guarantee a strict upper bound on the norm of the random products, bounded away from what is expected from the exponent. In particular, this implies an upper bound for the entropy of infinite sequences with bounded growth when the top Lyapunov exponent of the cocycle is positive; see, for instance, \cite[Chapter 14]{BQ}.

Finally, we note that while our results provide strict lower and upper bounds for the entropy of the set of fiberwise bounded orbits for certain open sets of cocycles, they do not determine its exact value. In light of this, it would be highly interesting to develop a multifractal analysis for this setting, similar to the analysis developed for Lyapunov exponents. This could yield explicit formulas or more delicate information, such as the continuity of the entropy or the Hausdorff dimension of the set of fiberwise bounded orbits with respect to the system’s parameters.

\subsection*{Organization of the paper} 
The paper is organized as follows. Section~\ref{sec:prelim} introduces the primary definitions and notation used throughout the paper. Section~\ref{sec:multiple-covering} reviews the covering condition for semigroup actions and linear cocycles, generalizing it to introduce the core concept of multiple covering. Furthermore, we establish the multiple covering principle within this section. Section~\ref{sec:entropy} presents the central argument demonstrating that the multiple covering condition yields a lower bound for the entropy of the set of fiberwise bounded orbits. Section~\ref{sec:rigidity} investigates rigidity phenomena achieved when the set of fiberwise bounded orbits of a certain bound attains full topological entropy.
 In Section~\ref{sec:proof-main-thm}, we employ the tools developed in the preceding sections  to complete the proof of Theorem~\ref{thm:A}. Finally, Section~\ref{sec:questions} concludes the paper with further remarks and open questions for future research.
\subsection*{Acknowledgement}The authors would like to thank Jairo Bochi and Salman Beigi for useful conversations and comments. Part of this work was written while  M.N. and Z.R. were visiting BICMR, Peking University. This work is partially supported by INSF grant no. 4041991.
\section{Preliminaries and Notations} 
\label{sec:prelim}
For every non-empty set $\cS$, we denote
$$
\cS^{\times n} := \underbrace{\cS \times \cdots \times \cS}_{n \text{ times}}, \quad \text{and} \quad \cS^{\mathrm{fin}} := \bigcup_{k \in \N} \cS^{\times k}.
$$

We say a word $\uw \in \cS^{\mathrm{fin}}$ has length $n$, denoted by $|\uw| = n$, to indicate that $\uw \in \cS^{\times n}$. 

Let $G$ be a group. For every $\uw = (g_1, \ldots, g_n) \in G^{\mathrm{fin}}$, we define \emph{the product map},
$$
\cP_{[\uw]} := g_n \cdots g_1 \in G.
$$
For any subset $\cW \subset G^{\mathrm{fin}}$, let $\cP_{\cW} := \{\cP_{[\uw]} : \uw \in \cW\}$. Furthermore, for non-empty subsets $\cS_1, \ldots, \cS_n \subset G$, we denote 
$$
\cS_n \cdots \cS_1 := \cP_{\cS_n \times \cdots \times \cS_1}.$$
Consistent with this, for every $n\in \N$, every $\cS\subset G$, and $s\in G$, we set $\cS^n := \cP_{\cS^{\times n}}$, $\cS^{-1} := \{s^{-1} : s \in \cS\}$, and $s\cS := \{s\} \cS$.

\label{sec:pre}\subsection{Basics of linear cocycles} Let $d\in \N$ and $G$ be a subgroup of $\gldr$, acting linearly on $\R^d$. A \emph{$G$-cocycle} over a homeomorphism $f$ of a compact metric space $X$ is determined by a continuous map $A : X \to G$, which defines a skew-product homeomorphism $F : X \times \R^d \to X \times \R^d$ given by
\[F(x,v) = (f(x), A(x)v).\] 
We denote such a cocycle by $(f,A)$. For $n \in \Z$, the $n$-th iterate of the cocycle is denoted by $(f^n, A_n)$, where $A_0(x)$ is the identity matrix for all $x \in X$, and for $n \in \N$,
\begin{equation*}
    A_n(x) := A(f^{n-1}(x)) \cdots A(f(x)) A(x)\quad \text{and}\quad 
    A_{-n}(x) := \big(A_n(f^{-n}(x))\big)^{-1}.
\end{equation*}
Note that the continuity of the map $A$ is defined with respect to the metric on $\gldr$ 
given by
$
\|P_1-P_2\|+\|P_1^{-1}-P_2^{-1}\|
$, where $\|\cdot\|$ denotes the operator norm on matrices, $\|P\|:=\sup_{\|v\|=1} \|Pv\|$, induced by the standard Euclidean norm on $\R^d$.

\medskip

The main classes of cocycles we consider are those for which $G = \sldr$ or $G = \gldr$. One may consider both $\cC^0$ and  H\"older regularities and topologies on the space of linear cocycles over $f$. 
For $\alpha> 0$, the cocycle $(f,A)$ is an \emph{$\alpha$-H\"older cocycle} if there exists  $C>0$ such that for all $x,y \in X$, $\|A(x)-A(y)\| \leq C\, \dist(x,y)^{\alpha}$.   The $\cC^0$ topology on this space is the topology induced by the metric 

\begin{align}
    d_{\cC^0}(A,A')
:=
\sup_{x \in X}
\left(
\|A(x)-A'(x)\|
+
\|A(x)^{-1}-A'(x)^{-1}\|
\right).
\end{align}

This metric makes the space of all continuous linear cocycles over $X$ a Baire space. 

We also consider the H\"older topology on the space of $\alpha$-H\"older cocycles. To recall this, first, for every $\eta>0$, the local $\cC^{\alpha}$ semi-norm and local $\cC^{\alpha}$ norm of an $\alpha$-H\"older cocycle $A$ are defined by 
\begin{align}
 |A|_{\cC^{\alpha},\eta} &:=
\sup_{0 < \dist(x,y) < \eta}
\frac{
\|A(x)-A(y)\|
+
\|A(x)^{-1}-A(y)^{-1}\|
}
{\dist(x,y)^{\alpha}},\\~& \notag\\
\|A\|_{\cC^{\alpha},\eta} &:= \|A\|_{\cC^{0}} + |A|_{\cC^{\alpha},\eta}. 
\end{align}
Next, the global $\cC^{\alpha}$ semi-norm and norm are defined by 
\begin{align}
     |A|_{\cC^{\alpha}} & := |A|_{\cC^{\alpha}, \mathrm{diam}(X)},\\
\|A\|_{\cC^{\alpha}} & := \|A\|_{\cC^{0}} + |A|_{\cC^{\alpha}}. 
\end{align}

One can verify that for every pair of $\alpha$-H\"older cocycles $(f,A), (f,B)$, the inequality $|A+B|_{\cC^{\alpha}} \leq |A|_{\cC^{\alpha}} + |B|_{\cC^{\alpha}}$ holds. Hence, the norm $\|\cdot\|_{\cC^{\alpha}}$ induces a metric on $\cC^{\alpha}$ given by 
\begin{align}
    d_{\cC^{\alpha}}(A,B) := \|A-B\|_{\cC^{\alpha}}.
\end{align}

\begin{remark}
We remark that the $\cC^0$ and the H\"older metrics on the space of linear cocycles used in the present paper defer apparently from the ones used in \cite{NRR1}, where the norms were defined without the inverse terms. However, the topologies induced by the corresponding metrics are the same. Indeed, for a given continuous cocycle $(f,A)$,
\[\|A(x)^{-1}-A(y)^{-1}\|\leq  \|A(x)^{-1}\| . \|A(y)^{-1}\|.\|A(x)-A(y)\|\leq K^2 \|A(x)-A(y)\|,\]
where $K=\sup_{x\in X} \|A(x)^{-1}\|$. 
Hence, adding the term $\|A(x)^{-1}-A(y)^{-1}\|$ to the  $\cC^r$ norm and metric ($r\geq 0$) yields a locally equivalent norm and metric, respectively. Therefore, all results from \cite{NRR1} stated with respect to the usual H\"older topology remain valid for the norm used here.
\end{remark}

\subsection{Cocycles admitting fiberwise bounded orbits}\label{sec:qc-and-bounded}
We say that a cocycle $(f,A)$ over a homeomorphism $f: X \to X$ \emph{admits fiberwise bounded orbits} if there exists a point $x \in X$ such that the norms of the matrices $A_n(x)$ and their inverses along the full orbit of $x$ remain uniformly bounded. That is, 
\begin{equation}\label{eq:b-property-def}
    \sup_{n \in \Z}  \|A_n(x)^{\pm 1}\| < +\infty. 
\end{equation}

Note that for $\sldr$ cocycles, the condition \eqref{eq:b-property-def} is equivalent to the boundedness of the sequence $\{\|A_n(x)\|\}_{n\in \Z}$. We denote the set of all cocycles admitting fiberwise bounded orbits by $\cB$. 
In certain contexts, it is necessary to explicitly emphasize the upper bound of the supremum in \eqref{eq:b-property-def}. Motivated by this, for any $\kappa\geq 0$, we define
\begin{align*}
    &\mathrm{B}^+(A,\kappa) := \bigl\{ x \in X : \|A_n(x)^{\pm 1}\| \le \kappa \text{ for all } n \in \N \bigr\},\\
    &\mathrm{B}^{\pm}(A,\kappa) := \bigl\{ x \in X : \|A_n(x)^{\pm 1}\| \le \kappa \text{ for all } n \in \Z \bigr\},
\end{align*}
Since $\mathrm{B}^{\pm}(A,\kappa)$ is not necessarily
$\sigma$-invariant, we define ${\bfB}(A,\kappa)$ to be its maximal
$\sigma$-invariant subset.

The following lemma provides some primary properties of the sets $\mathrm{B}^+(A,\kappa)$. Analogous properties hold for $\mathrm{B}^{\pm}(A,\kappa)$, the set of points with fiberwise bounded full orbit, in which $\mathrm{B}^+(A,\kappa)$ is replaced by its corresponding full-orbit counterpart.

\begin{lemma}\cite[Lemma 3.3]{NRR1}\label{lem:b-properties}
For every  $\kappa,\kappa' > 0$, 
\begin{itemize}
    \item[(i)] If $\|A_m(x)^{\pm 1}\| \le \kappa$ and $\|A_n(x)^{\pm 1}\| \le \kappa'$,  for some  $n>m\geq 0$, then $$\|A_{n-m}(f^m(x))\| \le \kappa\kappa'.$$
    \item[(ii)] $\mathrm{B}^+(A,\kappa)$ is closed.
    \item[(iii)] If $x \in \mathrm{B}^+(A,\kappa)$, then $\cO^+(x) \subset \mathrm{B}^+(A,\kappa^2)$.
    \item[(iv)] If $x \in \mathrm{B}^+(A,\kappa)$, then  $ \omega(x)\subset {\bfB}(A,\kappa^2)$.
\end{itemize}
\end{lemma}

A corollary of this lemma is that 
\begin{align}\label{eq:inclusion-B}
     \mathrm{B}^{\pm}(A,\kappa)\subset \mathrm{B}^+(A,\kappa)\subset {\bfB}(A,\kappa^2) \subset \mathrm{B}^{\pm}(A,\kappa^2).
\end{align}

A very related notion to the set of fiberwise bounded orbits is that of quasi-conformal orbits. We say that  a cocycle $(f,A)$ \emph{admits quasi-conformal orbits} if there exists $x \in X$ such that the sequence of condition numbers $\{ \|A_n(x)\| \|A_n(x)^{-1}\| \}_{n\in \Z}$ is bounded. See \cite[Section 3]{NRR1} for a comprehensive discussion.

\subsection{Symbolic dynamics}
\label{subsec:symbolic-dyn}

Let $\cA$ be a finite alphabet. For simplicity, we usually assume $\cA:=\{1,2,\ldots,|\cA|\}$.  The set of all bi-infinite sequences over $\cA$ is denoted by $\Sigma_\cA$ or $\cA^{\times \Z}$. 
For integers $n \le n'$ and a finite word $\ua=(a_{n},a_{n+1},\ldots,a_{n'})\in \cA^{\times (n'-n+1)}$, the corresponding cylinder set $\cylinder[n;\ua]\subset \Sigma_\cA$ is defined as
$$
\cylinder[n;\ua]:=\bigl\{(w_i)_{i\in \Z}\in {\Sigma_\cA}\;  :\;  w_i=a_i ~ \text{for}~ i=n,\ldots,n'\bigr\}.
$$
If $\ua\in \cA^{\times (2n+1)}$ for some $n\in \N\cup \{0\}$, we simply write $\cylinder[\ua]$ instead of $\cylinder[-n;\ua]$. These cylinder sets form a basis for the topology on $\Sigma_\cA$, which is induced by the metric
\begin{equation}\label{eq:dist-shift-space}
\dist\big((x_i)_{i\in \Z},(y_i)_{i\in \Z}\big)=2^{-n},\quad \text{where}~~ n= \min \{|i|: x_i \neq y_i \}.
\end{equation} 
With this topology, the left shift map $\sigma:\Sigma_\cA\to \Sigma_\cA$, defined by
$$
\sigma(\ldots,x_{-1};x_0,x_1,\ldots):=(\ldots,x_0;x_1,x_2,\ldots),
$$
is a homeomorphism. For any $x=(x_i)_{i\in \Z}\in \Sigma_\cA$, the local unstable set is given by
 \begin{align*}
     W_{loc}^u(x) &:=\bigl\{(y_i)_{i\in \Z} :  y_i=x_i~ \text{for}~ i\leq  0\bigr\}.
 \end{align*}
 The local stable set, $W^s_{loc}(x)$ is defined similarly. 
Furthermore, for any $k\in \N$ and finite word $\underline{x}=(x_1,\ldots,x_k)\in \cA^{\times k}$, we can naturally associate $\underline{x}$ with the periodic point $(\ldots, \underline{x};\underline{x},\underline{x},\ldots)$ of $\sigma$ (of period $k$).

A particularly important class of shift subsystems is that of \emph{subshifts of finite type} (SFTs). Let $H$ be an $|\cA| \times |\cA|$ transition matrix with entries in $\{0,1\}$. For $n \ge 1$, we define the set of \emph{$H$-admissible words of length $n$} by$$\cA_H^{\times n} := \bigl\{ (x_1, \dots, x_n) \in \cA^{\times n} : H_{x_i, x_{i+1}} = 1 \text{ for all } 1 \le i \le n-1 \bigr\}.$$Furthermore, let $\cA_H^{\mathrm{fin}} := \bigcup_{n \ge 1} \cA_H^{\times n}$. The set of bi-infinite $H$-admissible words is defined analogously and is denoted by $\Sigma_H$, which is a closed, $\sigma$-invariant subset of $\Sigma_\cA$. The corresponding SFT is defined as the restriction of $\sigma$ to this set.

The topological entropy of the restriction of $\sigma$ to $\Sigma_H$ can be explicitly calculated as $h_{\mathrm{top}}(\sigma) = \log   \rho(H)$, where $\rho(H)$ denotes the spectral radius of $H$ (i.e., the maximum modulus of its eigenvalues). Since the matrix $H$ has non-negative entries, the Perron-Frobenius theorem guarantees that $\rho(H)$ is a real, positive eigenvalue, usually referred to as the Perron eigenvalue (or Perron-Frobenius eigenvalue). It is a well-known fact that a transitive SFT has strictly positive topological entropy if and only if it is not a single periodic orbit. Recall that an SFT is topologically transitive if it contains a dense forward orbit, which is equivalent to the  matrix $H$ being irreducible.

In general, for a closed subset $\Sigma \subset \cA^{\times \Z}$ with $\sigma(\Sigma)=\Sigma$ (not necessarily of finite type), the topological entropy of the restriction of the shift map $\sigma$ to $\Sigma$ is given by the exponential growth rate of the number of admissible words of length $n$. More precisely, let $\cL_n(\Sigma)$ denote the set of distinct words of length $n$ that appear as subwords of elements in $\Sigma$. Then
\begin{equation}\label{eq:entropy-growth-word}
     h_{\mathrm{top}}(\sigma \big|_{\Sigma})
    = \lim_{n \to \infty} \frac{1}{n} \log   |\cL_n(\Sigma)|,
\end{equation}
where the limit always exists due to the subadditivity of the sequence $\log  | \cL_n(\Sigma)|$.   For a more comprehensive background on symbolic dynamics and SFTs, we refer the reader to \cite{Lind_Marcus_Symbolic}.

\medskip

We also need to introduce some operations on $\cA^{\mathrm{fin}}_H$. For $\uw=(w_1,\ldots,w_n)\in \cA_H^{\times n}$ and a positive integer $t\leq n$, the \emph{$t$-prefix} and \emph{$t$-suffix} of $\uw$ are $(w_1,\ldots,w_t)\in \cA^{\times t}$ and $(w_{n-t+1},\ldots,w_n)\in \cA^{\times t}$, respectively. For two words $\underline{w}=(w_1,\ldots,w_n)\in \cA_H^{\times n}$ and $\underline{w}'=(w'_1,\ldots,w'_{n'})\in \cA_H^{\times n'}$, and an integer $0 \leq t \leq \min\{n,n'\}$, if the $t$-suffix of $\underline{w}$ coincides with the $t$-prefix of $\underline{w}'$, we define their \emph{$t$-amalgamated concatenation}, denoted by $\underline{w}\vee_{t} \underline{w}'\in \cA_H^{\times (n+n'-t)}$ (or simply $\underline{w}\vee \underline{w}'$ if there is no ambiguity), by 
\begin{equation}\label{eq:def-merge}
  \underline{w}\vee_{t} \underline{w}':=(w_1,\ldots,w_n,w'_{t+1},w'_{t+2},\ldots,w'_{n'})=(w_1,\ldots,w_{n-t},w'_1,\ldots,w'_{n'}).
\end{equation}
Note that when $t=0$, the condition on prefixes and suffixes is vacuous, and $\underline{w}\vee_0 \underline{w}'$ is exactly the concatenation of the two words, and we denote it alternatively by $\underline{w}\,\underline{w}'$. In this case, it is required that $H_{w_n,w'_1}=1$ to ensure $\underline{w}\,\underline{w}'\in \cA_H^{\mathrm{fin}}$. For a positive integer $k$, we use the abbreviation $\underline{w}^k$ to denote the $k$-fold concatenation $\underline{w}\,\underline{w} \cdots \underline{w}$ ($k$ times), provided it is admissible in $\Sigma_H$.

\medskip
In accordance with \cite[Definition 2.6]{NRR1}, for $\ua, \ub \in \cA_H^{\mathrm{fin}}$, a word $\uw \in \cA_H^{\mathrm{fin}}$ is called a \emph{transition from $\ua$ to $\ub$} if $\ua$ is the $|\ua|$-prefix of $\uw$ and $\ub$ is its $|\ub|$-suffix. The set of all transitions from $\ua$ to $\ub$ is denoted by $T(\ua, \ub)$.
Observe that for any $\ua, \ua', \ua'' \in \cA_H^{\mathrm{fin}}$, and $\uw \in T(\ua, \ua')$, $\uw' \in T(\ua', \ua'')$, we have
$$
\uw \vee_{|\ua'|} \uw' \in T(\ua, \ua'').
$$ 

Given an integer $n\geq 1$, one can associate a directed graph with the subshift $\sigma:\Sigma_H\to \Sigma_H$ as follows. Consider all the elements of $\cA_H^{\times n}$ as vertices, and draw a directed edge from $\ua$ to $\ub$ if the $(n-1)$-suffix of $\ua$ coincides with the $(n-1)$-prefix of $\ub$ (Figure \ref{fig:graph}). 
In this viewpoint, a transition between two elements of $\cA^{\mathrm{fin}}_H$ having the same length $n$ corresponds to a walk in the associated graph. 

\begin{figure}[htpb]
     \begin{subfigure}[b]{0.6\textwidth}
         \centering
         \includegraphics[width=\textwidth]{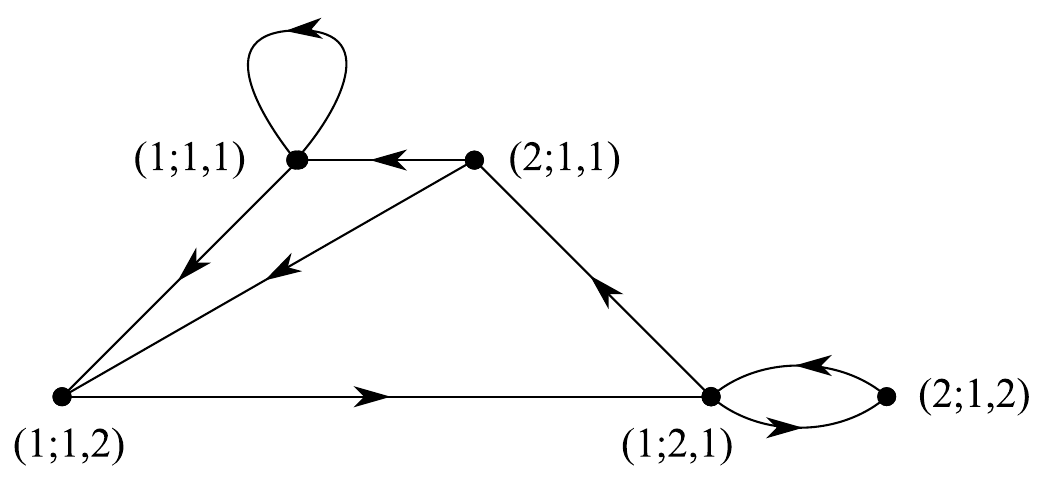}
     \end{subfigure}
    \caption{The graph associated with the SFT corresponding to $H=(\begin{smallmatrix}
        1&1\\
        1&0
    \end{smallmatrix})$, where $\cA=\{1,2\}$ and $n=3$.}
    \label{fig:graph}
\end{figure}

The next lemma establishes some properties of the aforementioned graph when the SFT is transitive.   

\begin{lemma}\label{lem:connected}
Let $\sigma:\Sigma_H\to \Sigma_H$ be a transitive SFT. Then for every $\ua,\ub\in \cA^{\mathrm{fin}}_H$, we have $T(\ua,\ub)\neq \emptyset$. In particular, for every $n\geq 1$, the graph associated with the vertex set $\cA_H^{\times n}$ is strongly connected (i.e., there is a directed path between every two vertices). Moreover, there exists a vertex with an out-degree of at least two, unless $\Sigma_H$ consists of a single periodic orbit of $\sigma$, or equivalently $h_{\mathrm{top}}(\sigma\big|_{\Sigma_H})=0$. 
\end{lemma}

\begin{proof}
Fix $\ua,\ub\in \cA^{\mathrm{fin}}_H$. Let $x=(x_i)_{i\in \Z} \in \Sigma_H$ be a point whose forward orbit is dense in $\Sigma_H$. Then there exist positive integers $m_1,m_2$ with $m_1 +|\ua| < m_2$ such that $\sigma^{m_1}(x)\in \cylinder[0;\ua]$ and $\sigma^{m_2}(x)\in \cylinder[0;\ub]$. Now, $\uw=(x_{m_1},x_{m_1+1},\ldots,x_{m_2+|\ub|-1})$ is a transition from $\ua$ to $\ub$. 

Finally, if for some $n\geq 1$, each vertex of the graph with vertices $\cA_H^{\times n}$ has an out-degree of one, strong connectivity of the graph directly implies that it must be a single directed cycle, which corresponds to a single periodic orbit and the topological entropy is clearly zero. 
\end{proof}

These notations are particularly useful when working with the iterates of locally constant cocycles over SFTs, a class of cocycles that appears frequently in this paper. For every locally constant $\gldr$-cocycle $(\sigma,A)$ over $\sigma:\Sigma_H\to \Sigma_H$, there exist integers $k,k'\geq 0$ and a map $\phi: \cA_H^{\times(k+k'+1)} \to \gldr$ such that
\[
A|_{\cylinder[-k;\ua]} = \phi(\ua), \quad 
\text{for every } \ua \in \cA_H^{\times (k+k'+1)}.
\]
In this setting, we say that $(\sigma, A)$ \emph{depends on positions in} 
$\ldbrack -k, k'\rdbrack := \{-k, \dots, k'\}$ and is \emph{represented by} $\phi$. 
For convenience, and without loss of generality, we usually assume the blocks are symmetric (i.e., $k=k'$).

In such a case, for every transition $\underline{w} = (w_1, w_2, \dots, w_n)$ between two elements of $\cA_H^{\times (2k+1)}$, we define 
\[
A_{[\underline{w}; 2k+1\rangle} := \phi(\underline{w}_{n-2k}) \cdots \phi(\underline{w}_2)\phi(\underline{w}_1),
\]
where $\underline{w}_i = (w_i, w_{i+1}, \dots, w_{i+2k})$. It is clear that for $\underline{w} \in T(\underline{a}, \underline{a}')$ and $\underline{w}' \in T(\underline{a}', \underline{a}'')$, we have
\begin{equation}\label{eq:composition-transition}
    A_{[\underline{w} \vee_{2k+1} \underline{w}'; 2k+1\rangle} = A_{[\underline{w}'; 2k+1\rangle} A_{[\underline{w}; 2k+1\rangle}.
\end{equation}
Indeed, it is easy to verify that $A_{[\underline{w}; 2k+1\rangle}$ is the iterate of the cocycle $(\sigma,A)$ over the cylinder $\cylinder[-k;\uw]$, and \eqref{eq:composition-transition} naturally reflects the cocycle property of $A$. 

Translated into the language of graphs, for a locally constant cocycle depending on positions in $\ldbrack -k, k\rdbrack$, one can consider the directed graph associated with the SFT whose vertices are the elements of $\cA_H^{\times (2k+1)}$. The matrix $\phi(\ua)$ is then assigned to the vertex $\ua$. Consequently, for any transition $\uw$, the matrix $A_{[\underline{w}; 2k+1\rangle}$ represents the product of the matrices associated with the vertices visited along the walk $\uw$, excluding the final vertex.

\section{From Covering to Multiple Covering}
\label{sec:multiple-covering}

In this section, we begin by recalling the \emph{covering condition} for topological groups (see \cite{FNR}). We then extend this concept to the notion of multiple covering, which is fundamental for achieving positive entropy within the set of fiberwise bounded orbits.

We say that a subset $\cG$ of a topological group $G$ satisfies the \emph{covering condition} with respect to a non-empty set $\cU \subset G$ if 
\begin{equation}\label{eq:covering-IFS}
\overline{\cU} \subset \cG^{-1}\cU.
\end{equation}

Furthermore, we say that a set $\cW\subset G^{\mathrm{fin}}$ satisfies the \emph{multiple covering condition of $n$ layers}  with respect to a non-empty set $\cV\subset G$ if there exist mutually disjoint subsets $\cW_1, \dots, \cW_n \subset \cW$ such that for every $j = 1, \dots, n$,
\begin{equation}
\overline{\cV} \subset (\cP_{\cW_j})^{-1} \cV.
\end{equation}

Note that condition \eqref{eq:covering-IFS} is not vacuous. Indeed, for any non-empty open set $\cU$, the condition can be satisfied by taking $\cG$ to be a sufficiently small neighborhood of the identity. Furthermore, if $\cU$ has a compact closure, a standard compactness argument allows $\cG$ to be replaced by a finite subset. About the multiple covering condition, while an obvious way to establish the multiple covering condition is to use several distinct sets $\cG_j$ that each satisfy the covering condition for the same set $\cU$, Lemma~\ref{lem:IFS-multi} below provides a more subtle argument. It shows that if $\overline{\cU}$ is compact, the covering condition  \eqref{eq:covering-IFS} for $\cG$ implies the multiple covering condition for one of its finite powers, $\cG^{\times k}$.

For our purposes, the group $G$ is a closed subgroup of $\gldr$, typically $\sldr$ or $\gldr$ itself, and the sets $\cU, \cV, \dots$ are non-empty, precompact, open subsets of $G$. Here, by a precompact set we mean a set whose closure in $G$ is compact. When $G$ is a closed subgroup of $\gldr$, the covering condition \eqref{eq:covering-IFS} plays a crucial role in identifying the fiberwise bounded orbits for the locally constant cocycle, which depends on the zeroth position naturally associated with $\cG$, as well as ensuring the stability of their existence. The following lemma from \cite{FNR} illustrates this relationship.

For a finite set $\cG \subset \gldr$, we say that $\langle \cG \rangle^+$, the semigroup generated by $\cG$, has a \emph{quasi-conformal branch} if there exists a sequence $\{D_i\}_{i=1}^\infty$ of elements in $\cG$ such that  $\sup_{k\in \N}\|D_k\cdots D_1\|.\|(D_k\cdots D_1)^{-1}\|<+\infty$.

\begin{lemma}\cite{FNR}\label{lem:U-branch}
Let $\cG \subset \operatorname{SL}(d,\R)$ be finite. Then, $\langle \cG \rangle^+$ has a quasi-conformal branch if and only if there exists a subset $\cU \subset \operatorname{SL}(d,\R)$ with compact closure such that $\overline{\cU} \subset \cG^{-1} \cU$.
\end{lemma}

For further discussion regarding this relationship, we refer the reader to \cite{FNR,NRR1}.

\begin{remark}\label{rmk:covering-improvement}
    We remark that for any Hausdorff topological group $G$ (such as $\gldr$ or its closed subgroups), if for a fixed subset $\cG$, the covering condition \eqref{eq:covering-IFS} holds for some non-empty open set $\cU$ with compact closure, then it also holds if we replace $\cU$ with another subset of the form $\cU\cZ$, provided that $\cZ$ is a non-empty set with compact closure. A similar property holds for the multiple covering condition. 

    The reason follows from a series of simple observations. First, if two subsets $\cZ_1, \cZ_2$ have compact closures, then $\overline{\cZ_1\cZ_2} = \overline{\cZ_1} \; \overline{\cZ_2}$. Second, whenever $\cO$ is open, then for any arbitrary subset $\cZ$, $\cO \, \cZ$ is again open, and moreover, $\cO \, \overline{\cZ} = \cO \, \cZ$. Using these observations, we conclude that when $\cG$ satisfies the covering condition \eqref{eq:covering-IFS} with respect to an open set $\cU$ with compact closure, then for any arbitrary non-empty set $\cZ$ with compact closure,
    \[
    \overline{\,\cU\, \cZ\,} = \overline{\cU}\, \overline{\cZ} \subset \cG^{-1}\,\cU \, \overline{\cZ} = \cG^{-1}\,\cU\, \cZ.
    \]
\end{remark}

The following lemma demonstrates that if a finite set $\cG$ satisfies the covering condition for a given open set $\cU$ with compact closure, we can guarantee a multiple covering of a new set $\cV$ by replacing $\cG$ with a sufficiently large $k$-fold product.

\begin{lemma}[Multiple covering principle for linear groups]\label{lem:IFS-multi}
Let $G$ be a closed subgroup of $\gldr$, and let $\cG\subset G$ be a non-empty set with compact closure, and containing least two elements. If $\cG$ satisfies the covering condition with respect to a non-empty precompact open subset $\cU$ of $G$, then for any $n \in \N$, there exist an integer $k \in \N$ and a non-empty precompact open set $\cV \subset G$ such that $\cG^{\times k}$ satisfies the multiple covering condition of $n$ layers with respect to $\cV$.
\end{lemma}

The proof of Lemma~\ref{lem:IFS-multi} relies on the following propositions regarding the robustness of the covering property.

\begin{proposition}\label{prop:uniform-delta}
Let $G\subset \gldr$ be a closed subgroup and $\cG,\cU$ be non-empty subsets of $G$. If $\cU$ is precompact and open with $\overline{\cU} \subset \cG^{-1}\cU$, then for every interval $[R_{-}, R_{+}]\subset \R_{>0}$, there exists $\delta>0$ such that for every $R \in [R_{-}, R_{+}]$,
\begin{equation}\label{eq:covering-bigger-smaller}
    \cU\, \cB_{R+\delta} \subset \cG^{-1} \, \cU\,\cB_{R-\delta},
\end{equation}
where, for $r>0$, $
\cB_r
:=
\left\{
g\in G :
\|g-\id\|+\|g^{-1}-\id\|<r
\right\}
$. 
\end{proposition}

\begin{proof}  Fix an arbitrary $R>0$. We claim that there exists some $\delta_R>0$ such that inclusion \eqref{eq:covering-bigger-smaller} holds for all $\delta \in (0, \delta_R]$. 
In view of Remark \ref{rmk:covering-improvement}, the covering condition \eqref{eq:covering-IFS} implies 
$$ \overline{\,\cU \, \cB_R\,} \subset \cG^{-1} \, \cU \, \cB_R. $$
Observe that as $\delta \searrow 0$, the sets $\{\cG^{-1} \cU \cB_{R-\delta}\}_{\delta>0}$ form an increasing family of open sets with union $\cG^{-1} \cU \cB_{R}$. Conversely, $\{\overline{\cU\cB_{R+\delta}}\}_{\delta>0}$ is a decreasing family of compact sets with intersection 
$ \overline{\cU\, \cB_{R}}$. 
Since the intersection of this descending compact family is contained within the union of the ascending open family, a standard compactness implies that there exists $\delta_R>0$ (smaller than $R$) such that \eqref{eq:covering-bigger-smaller} holds for this specific $R$ whenever $\delta \leq \delta_R$. 
Note that this choice of $\delta_R$ has a local uniformity. Indeed, for every $R'\in I_R := (R-\frac{1}{2}\delta_R, R+\frac{1}{2}\delta_R)$ and any positive $\delta' \leq \frac{1}{2}\delta_R$,
$$ \cU\, \cB_{R'+\delta'} \subset \cU\, \cB_{R+\delta_R} \subset \cG^{-1} \, \cU\,\cB_{R-\delta_R} \subset \cG^{-1} \, \cU\,\cB_{R'-\delta'}. $$

Finally, the intervals $I_R$ for $R\in [R_{-}, R_{+}]$ form an open cover of the compact set $[R_{-}, R_{+}]$ and so has a finite subcover corresponding to $R_1, \dots, R_m$. Therefore, taking $\delta := \min_{1 \leq i \leq m} \frac{1}{2}\delta_{R_i} > 0$ completes the proof of the proposition.
\end{proof}

\begin{proposition}\label{prop:cover-smaller-set}
Let $G$ be a closed subgroup of $\gldr$ and  $\cU,\cS$ be non-empty subsets of $G$ with compact closure. Then, for every $r>0$, there exists $\eta=\eta(\cU,\cS,r)>0$ such that for every $s\in \cS$, 
$\cU\cB_{r} \subset s \cU\cB_{\eta}$.
\end{proposition}

\begin{proof}
Note that for any $b\in \cB_r$, we have $\|b-\id\| \le \|b-\id\| + \|b^{-1}-\id\| < r$, which implies $\|b\|, \|b^{-1}\| < r+1$.
   Now, take $\beta:=\sup_{g\in \cS\cup \cU} \|g^{\pm 1}\|$ and $\eta$ greater than $2 + 2(r+1)\beta^3$.  Therefore, for every $u, u' \in \cU$, every $s \in \cS$, and every $b \in \cB_{{r}}$,  we can bound the distance of $u^{-1}s^{-1}u'b$ and its inverse to the identity:
\begin{align*}
\|u^{-1}{s}^{-1}u' b - \id\| &\leq 1+ \|u^{-1}\|\|{s}^{-1}\|\|u'\| \|b\| < 1 + \beta^3(r+1), \\
\|(u^{-1}{s}^{-1}u' b)^{-1} - \id\| &= \|b^{-1}(u')^{-1}s u - \id\| \leq 1 + \|b^{-1}\|\|(u')^{-1}\|\|s\| \|u\| < 1 + \beta^3(r+1).
\end{align*}
Summing these inequalities yields
$$
\|u^{-1}{s}^{-1}u' b - \id\| + \|(u^{-1}{s}^{-1}u' b)^{-1} - \id\| < 2 + 2\beta^3(r+1) \leq \eta.
$$
Consequently, $u^{-1}s^{-1}u'b \in \cB_{\eta}$, which implies $u'b \in su\cB_{\eta} \subset s\cU \cB_{\eta}$. Finally, as $u'\in \cU$ and $b\in \cB_{r}$ were chosen arbitrarily, $\cU \cB_{r} \subset s\cU \cB_{\eta}$ follows.
\end{proof}

\begin{proof}[{Proof of Lemma \ref{lem:IFS-multi}}]
Since $\cG$ contains at least two elements, we can choose an integer $k_0 \in \N$ such that $|\cG^{\times k_0}| \geq n$. Next, we fix $R_0$ to be a sufficiently large number. Specifically, in the notation of Proposition \ref{prop:cover-smaller-set}, it suffices to pick $R_0 > \eta(\cU, (\cP_{\cG^{\times k_0}})^{-1}, 1)$. 

Applying Proposition~\ref{prop:uniform-delta} to the interval $[1, R_0]$ yields a constant $\delta > 0$ such that for every $R$ in this interval, we have $\cU \cB_{R+\delta} \subset \cG^{-1}\cU \cB_{R-\delta}$. This immediately implies  $\overline{\cU \cB_R} \subset \cG^{-1}\cU \cB_{R-\delta}$. Now, let $m$ be the smallest integer such that $R_0 - m\delta < 1$. It is straightforward to verify by induction on $j$ that 
\begin{equation}\label{eq:better-cover}
   \overline{\cU \cB_{R_0}} \subset (\cP_{\cG^{\times j}})^{-1} \cU \cB_{R_0 - j\delta}, \qquad  \text{for every } j\in \ldbrack 1, m\rdbrack.
\end{equation}
By the choice of $R_0$ and Proposition \ref{prop:cover-smaller-set} (for $\cS=(\cP_{\cG^{\times k_0}})^{-1}$ and $r=1$), we obtain that for every $\uw\in \cG^{\times k_0}$, $\cU\cB_1\subset (\cP_{[\uw]})^{-1}\cU \cB_{R_0}$. Since $R_0-m\delta<1$, it follows that 
\begin{equation}\label{eq:better-cover-m}
\cU \cB_{R_0 - m\delta}\subset \cU \cB_1 \subset (\cP_{[\uw]})^{-1}\cU \cB_{R_0}. 
\end{equation}
Multiplying (from left) both sides of \eqref{eq:better-cover-m}  by $(\cP_{\cG^{\times m}})^{-1}$ and using \eqref{eq:better-cover}, we deduce that for every $\uw\in \cG^{\times k_0}$, and  denoting $\cW_{\uw} := \{\uw\} \times \cG^{\times m}\subset \cG^{\times (k_0+m)}$,
$$ \overline{\cU \cB_{R_0}} \subset (\cP_{\cG^{\times m}})^{-1}\cU \cB_{R_0 - m\delta} \subset (\cP_{\cG^{\times m}})^{-1}(\cP_{[\uw]})^{-1}\cU\cB_{R_0} = (\cP_{\cW_{\uw}})^{-1} \cU \cB_{R_0}. $$

Because the subsets $\cW_{\uw}$ for distinct $\uw \in \cG^{\times k_0}$ are mutually disjoint, and there are at least $n$ such $\uw$, the proof of the lemma is complete by setting $\cV := \cU \cB_{R_0}$ and $k := m + k_0$.
\end{proof}

\subsection{Multiple Covering Condition for Cocycles}

In this subsection, we extend the notions of the covering condition and multiple covering to general locally constant cocycles over SFTs that do not necessarily depend on the zeroth position.

For $\uw_1, \uw_2 \in \cA_H^{\mathrm{fin}}$, we write $\uw_1\prec \uw_2$ if $\uw_1$ is a prefix of $\uw_2$ with strictly smaller length, that is, there exists $\uw\in \cA_H^{\mathrm{fin}}$ such that 
$\uw_1 \uw = \uw_2$. A \emph{chain} (of length $k$) is a sequence $\uw_1,\ldots,\uw_k \in \cA_H^{\mathrm{fin}}$ satisfying $\uw_1\prec\cdots\prec\uw_k$. A subset $\cS\subset \cA_H^{\mathrm{fin}}$ is called an \emph{anti-chain} if it does not contain any pair $\uw_1,\uw_2$ with $\uw_1\prec\uw_2$. 

\begin{definition}[Multiple covering]\label{def:mutiple-covering}
    Let $m\in \N$ and $\sigma:\Sigma_H\to \Sigma_H$ be a subshift of finite type, and let $G$ be a closed subgroup of $\gldr$ with $d\geq 2$. 
    We say that a locally constant $G$-cocycle $(\sigma,A)$ depending on positions in $\ldbrack -k,+k\rdbrack$ satisfies the \emph{multiple covering condition of $m$ layers}
    if there exist:
    \begin{itemize}
        \item a finite family $\cU_1,\cU_2,\ldots,\cU_n$ of non-empty precompact open subsets of $G$,
        \item a subset $\cA'\subset \cA_H^{\times({2k+1})}$,
        \item non-empty subsets $\Lambda^{\jmath}_\imath(\ua) \subset \bigcup\limits_{\ub \in \cA'}T(\ua,\ub)$ for $\ua \in \cA'$, $\imath\in \ldbrack 1,n\rdbrack$, and $\jmath \in \ldbrack 1,m\rdbrack$,
    \end{itemize}
   such that the following properties hold:
    \begin{itemize}
        \item for every $\ua\in \cA'$ and $\imath\in \ldbrack 1, n\rdbrack$, the union $\bigcup_{\jmath=1}^m \Lambda^\jmath_\imath(\ua)$ is an anti-chain;
        \item for every $\ua\in \cA'$, $\imath\in \ldbrack 1,n\rdbrack$, and $\jmath\in \ldbrack 1,m\rdbrack$:
    \begin{equation}\label{eq:multigen-covering}
        \overline{\cU_\imath} \subset \bigcup\limits_{\uw\in \Lambda^\jmath_\imath(\ua)}(A_{[\uw\rangle})^{-1}\cU,\qquad \text{where  }~~~\cU:=\bigcup\limits_{\imath =1}^n \cU_{\imath}.
    \end{equation} 
    \end{itemize}
    Furthermore, we say that the cocycle satisfies the \emph{multiple covering condition} if it satisfies the multiple covering condition of $m$ layers for some $m>1$.
\end{definition}
Unlike the multiple covering condition for linear groups, this condition is inherently piecewise, allowing the choice of maps to depend on the base point. This localized flexibility is essential for adapting the multiple covering condition for cocycles, and yet strong enough for applications. 

Furthermore, as in Remark \ref{rmk:covering-improvement}, we may replace $\{\cU_\imath\}_{\imath}$ with $\{\cU_\imath\cZ\}_{\imath}$ for any non-empty precompact $\cZ \subset G$. Since $\overline{\cU_\imath}$ is compact, we can also replace $\cU$ in \eqref{eq:multigen-covering} with a smaller open set $\cU'$ satisfying $\overline{\cU'} \subset \cU$. Finally, setting $m=n=1$ recovers the covering condition from \cite{NRR1}:

\begin{definition}[{\cite[Definition 4.6]{NRR1}}]\label{def:gen-cov}
In the setting of Definition \ref{def:mutiple-covering}, we say that $(\sigma,A)$ satisfies the \emph{covering condition} if there exist a non-empty precompact open subset $\cU$ of $G$ and a subset $\cA' \subset \cA_H^{\times(2k+1)}$ such that for every $\underline{a} \in \cA'$,
\begin{equation}\label{eq:gen-covering}
      \overline{\cU} \subset \bigcup\limits_{\ub \in  \cA'}\bigcup\limits_{\uw\in T(\ua,\ub)}(A_{[\uw\rangle})^{-1}\cU.
\end{equation} 
\end{definition}

We conclude this section by establishing the equivalence between this covering condition and the multiple covering condition. In the next section we show that these conditions ensure the presence of fiberwise bounded orbits with positive topological entropy, and its stability.

\begin{theorem}[Multiple covering principle]\label{thm:multi-cov}
Let $d\in \N$ and $G$ be a closed subgroup of $\gldr$. 
    A locally constant $G$-cocycle $(\sigma,A)$ over a subshift of finite type of positive topological entropy satisfies the covering condition if and only if it satisfies the multiple covering condition of $m$ layers for some $m>1$. 
\end{theorem}

\begin{proof}
    If the locally constant $G$ cocycle $(\sigma,A)$ over the SFT $\sigma:\Sigma_H\to \Sigma_H$ satisfies the multiple covering condition, then it clearly satisfies the covering condition by setting $\cU=\bigcup_{\imath=1}^n\cU_\imath$. 
    To prove the converse, assume that $(\sigma,A)$ depends on positions in $\ldbrack -k,+k\rdbrack$ and satisfies the covering condition with $\cA'\subset \cA_H^{\times(2k+1)}$ and $\cU\subset G$. Throughout the remainder of the proof, we denote $\vee_{2k+1}$ simply by $\vee$. Also, $A_{[\uw\rangle}$ means $A_{[\uw;2k+1\rangle}$. For every $\ua,\ub \in \cA'$, let $\cT_{\ua,\ub}:=\{A_{[\uw\rangle}:\uw \in T(\ua,\ub)\}\subset G$ and
\begin{align}
    \cT_{\ua,\cA'}:=  & \bigcup_{\ub\in \cA'} \cT_{\ua,\ub} 
    \\     \label{eq:concat-transition}
    {\supset}  &\bigcup_{\ub\in \cA'} \cT_{\ub,\cA'} \, \cT_{\ua,\ub}.
\end{align}
Also, we need to consider the directed graph associated with $\sigma:\Sigma_H\to \Sigma_H$ with vertices corresponding to $\cA_{H}^{\times(2k+1)}$. Lemma \ref{lem:connected} ensures the following properties:
    \begin{itemize}
        \item This graph contains a vertex with out-degree of at least two. Let this vertex be associated with $\ud_0\in \cA_{H}^{\times(2k+1)}$, with directed edges to vertices associated with $\ud_1, \ud_2$ (note that these $\ud_i$ are not necessarily in $\cA'$).
        \item The graph is strongly connected. Let $\ell$ denote its directed diameter.
    \end{itemize}

    For every $\ua \in \cA'$ and $i=1,2$, we fix a  transition path from $\ua$ to $\ud_0$, followed by the edge from $\ud_0$ to $\ud_i$, and finally a transition from $\ud_i$ to some element of $\cA'$. Let $\uw^i(\ua)$ denote the concatenation of these paths. Clearly, it is an element of $T(\ua,\cA')=\bigcup_{\ub\in \cA'} T(\ua,\ub)$ with length bounded by $2\ell+1$.

Let $\cS$ be the set of all elements $s$ of $G$ with $\|s^{\pm 1}\|\leq \|A^{-1}\|^{2\ell+1}$.
    Applying Proposition \ref{prop:uniform-delta} to $\cG=\cT_{\ua,\cA'}$ and $\cU$ over the interval $[1,R_0]$, with some arbitrary fixed $R_0>\eta(\cU,\cS,1)>0$ (with notations of Proposition \ref{prop:cover-smaller-set}), we get a positive constant $\delta_{\ua}$. Defining $\delta:= \min_{\ua \in \cA'} \{ \delta_{\ua} \}$ guarantees that for every $\ua\in \cA'$ and $R\in [1,R_0]$,
    \begin{equation}\label{eq:n=1-first}
       {  \overline{\cU \cB_{R}}\subset  \cT_{\ua,\cA'}^{-1}\cU \cB_{R-\delta}.}
    \end{equation}
    Similar to the argument in the proof of Lemma \ref{lem:IFS-multi}, let $m$ be the smallest integer such that $R_0-m\delta <1$. By induction on $j$, we claim that
    \begin{equation}\label{eq:n=1-induction}
        \overline{\cU \cB_{R_0}}\subset  \cT_{\ua,\cA'}^{-1}\cU \cB_{R_0-j\delta}, \qquad \text{for every } \ua\in \cA' \text{ and } j\in \ldbrack1,m\rdbrack.
    \end{equation}
    Indeed, assuming the inclusion holds for $1\leq j<m$, we have $R_0-j\delta\in [1,R_0]$. Thus, for every $\ub\in \cA'$, ${\cU \cB_{R_0-j\delta}}\subset \cT_{\ub,\cA'}^{-1} \cU \cB_{R_0-(j+1)\delta}$. Consequently, by \eqref{eq:concat-transition},
    $$ 
\begin{aligned}
    \overline{\cU \cB_{R_0}} 
    &\subset \cT_{\ua,\cA'}^{-1} \, \cU \cB_{R_0-j\delta} 
    = \bigcup_{\ub\in \cA'} \cT_{\ua,\ub}^{-1} \, \cU \cB_{R_0-j\delta} \\
    &\subset \bigcup_{\ub\in \cA'} \cT_{\ua,\ub}^{-1} \, \cT_{\ub,\cA'}^{-1} \, \cU \cB_{R_0-(j+1)\delta} \\
    &{\subset
     \cT_{\ua,\cA'}^{-1} \, \cU \cB_{R_0-(j+1)\delta}.}
\end{aligned}
$$
In particular, for every $\ua\in \cA'$,  $\overline{\cU \cB_{R_0}}\subset \cT^{-1}_{\ua, \cA'}\cU \cB_{R_0-m\delta}\subset \cT^{-1}_{\ua, \cA'}\cU \cB_{1}$. 
This implies that for every  $\ua \in \cA'$ and every $v\in \overline{\cU\cB_{R_0}}$, there exist $\ub\in \cA'$ and  $\uw\in T(\ua,\ub)$ such that
$A_{[\uw\rangle} v\in \cU \cB_1$.  Since $\|A_{[\uw^i(\ub)\rangle}^{-1}\|\leq \|A^{-1}\|^{2\ell+1}$, the choice of $R_0$ allows us to conclude that by Proposition \ref{prop:cover-smaller-set},
    $$ A_{[\uw\rangle} v\in \cU \cB_1\subset  A_{[\uw^i(\ub)\rangle}^{-1} \cU\cB_{R_0}.$$
Equivalently, denoting $\cU\cB_{R_0}$ by $\cV$, we have $v\in A_{[\uw\vee\uw^i(\ub) \rangle}^{-1} \cV$. 
By the construction of the transition paths $\uw^i(\ub)$, the set $\{\uw\vee\uw^1(\ub),\uw\vee \uw^2(\ub)\}$ forms an anti-chain. Since the choice of $\uw$ depends on both $\ua$ and $v$, we define $\uw^{i}(\ua,v):=\uw\vee\uw^i(\ub)$, for $i=1,2$.  

Because $A_{[\uw^{i}(\ua,v) \rangle}^{-1} \cV$ is an open set containing $v$, there exists an open neighborhood $\cV_v$ of $v$ such that
$$ \overline{\cV_v} \subset A_{[\uw^{i}(\ua,v)\rangle}^{-1} \cV, \quad \text{for } i=1,2. $$

Finally, since $\overline{\cV}$ is compact, we can extract a finite subcover from the neighborhoods $\cV_v$ associated with finitely many points $v_1,\ldots,v_n \in \overline{\cV}$. We denote these open neighborhoods by $\cU_1,\ldots,\cU_n$. Now, the inclusion above can then be recast as follows: for every $\ua\in \cA'$, $\imath \in \ldbrack 1,n\rdbrack$, and $\jmath\in \ldbrack 1, 2 \rdbrack$, we have 
$$ \overline{\cU_\imath}\subset (A_{[\uw^\jmath (\ua,v_\imath)\rangle})^{-1}\cV \subset (A_{[\uw^\jmath (\ua,v_\imath)\rangle})^{-1}(\cU_1\cup \cdots \cup \cU_n). $$
Thus, the requirements of the multiple covering condition (Definition \ref{def:mutiple-covering}) are satisfied for $m=2$ using the sets $\cU_1,\ldots,\cU_n$, the subset $\cA'$, and the singletons $\Lambda^\jmath_{\imath}(\ua):=\{\uw^\jmath (\ua,v_\imath)\}$. This completes the proof.
\end{proof}

\section{The entropy of bounded orbits} \label{sec:entropy}

The goal of this section is to prove Theorem \ref{thm:C} which asserts that as a consequence of the multiple covering condition introduced in Section \ref{sec:multiple-covering}, the set of points with fiberwise bounded full orbit carries positive topological entropy.

To prove this theorem, we follow an approach analogous to \cite[Theorem 5.5]{NRR1} to establish the existence of a certain set of fiberwise bounded orbits. The proof uses the multiple covering condition, which allows multiple choices for extending a finite sequence. This branching property allows us to embed a full shift into the space of fiberwise bounded orbits. The construction of this embedding proceeds inductively and is based on the following proposition.

\begin{proposition}\label{prop:stability-H\"older-calculation}
    In the setting of Theorem \ref{thm:C} for $G$ and $(\sigma,A)$, if $(\sigma,A)$ satisfies the multiple covering condition of $m$ layers for some $m>1$, then there exists $\varepsilon>0$ such that for every $\alpha$-H\"older map $\tilde{A}:\Sigma_H\to G$ with $d_{\cC^\alpha}(A,\tilde{A})<\varepsilon$, we can find:
     \begin{itemize}
        \item a positive integer $L$,
        \item a precompact non-empty open set $\cV\subset G$, 
        \item a sequence $\{n_{i,j}\}$ of integers indexed by $i \in \N\cup\{0\}$ and $j\in \ldbrack 0, m^i-1\rdbrack$, and 
        \item a sequence $\{x_{i,j}\}$ of elements of $\Sigma_H$ indexed by $i \in \N\cup \{0\}$ and $j\in \ldbrack 0, m^i-1\rdbrack$,
     \end{itemize}
       such that for all $i\geq 0$ and $j\in \ldbrack 0, m^i-1\rdbrack$, letting $W_{i,j}:=W^u_{loc}\big(\sigma^{n_{i,j}}({x_{i,j}})\big)$, the following properties hold:
\begin{itemize}
    \item[(i)] (Nestedness and bounded gaps) For every $j'$ with $m j\leq j'< m(j+1)$,
    $$
    0<n_{i+1,j'}-n_{i,j}\leq L, \quad \text{and}\quad \sigma^{-n_{i+1,j'}}(W_{i+1,j'}) \subset \sigma^{-n_{i,j}}(W_{i,j}).
    $$
    \item[(ii)] (Disjointness) For every fixed $i\geq 0$, the sets $\sigma^{-n_{i,j}}(W_{i,j})$ for $j\in \ldbrack 0 , m^i-1\rdbrack$ are mutually disjoint.
   \item[(iii)] (Fiberwise boundedness at certain iterates) For every $y\in \sigma^{-n_{i,j}}(W_{i,j})$, there exists a sequence of times $0 = t_0 < t_1 < \ldots < t_i = n_{i,j}$ such that for every $l=1,\ldots,i$ we have $t_l - t_{l-1} \leq L$ , and moreover 
$$ \tilde{A}_{t_l}(y)\in \cV. $$
\end{itemize}
\end{proposition}

The proof of this proposition follows the same line of argument as \cite[Proposition 5.6]{NRR1}. We provide the main steps here for completeness. Unlike the aforementioned proposition, which was established for cocycles over full shifts, adapting the argument to subshifts of finite type and ensuring a positive entropy bound requires the use of double indices.

\begin{proof} 
Suppose that the cocycle $(\sigma,A)$ depends on positions in $\ldbrack -k,+k\rdbrack$ and is represented by a map $\phi:\cA^{\times(2k+1)}\to G$. Furthermore, assume that $(\sigma,A)$ satisfies the multiple covering condition of $m$ layers with respect to precompact open sets $\cU_1,\ldots,\cU_n\subset G$ with union $\cU$, a subset $\cA'\subset \cA_H^{\times(2k+1)}$, and families of words $\Lambda^{\jmath}_\imath(\ua)$ for $\jmath\in\ldbrack 1,m\rdbrack$ and $\imath\in \ldbrack 1, n\rdbrack$. Note that by the compactness of the sets $\cU_\imath$, we may assume that each $\Lambda^{\jmath}_\imath(\ua)$ is

finite. Additionally, by Remark \ref{rmk:covering-improvement}, we may suppose that each $\cU_\imath$ contains the identity element. 

Fix some $R>0$. In view of Proposition \ref{prop:uniform-delta}, there is a constant $\delta>0$ such that for every $\ua\in \cA'$, every $\imath \in \ldbrack 1,n\rdbrack$, and every $\jmath \in \ldbrack 1,m\rdbrack$, we have
    \begin{equation}\label{eq:smaller-u-t}
        \overline{\cU_\imath \cB_R} \subset \bigcup\limits_{\uw\in \Lambda^{\jmath}_\imath(\ua)}(A_{[\uw\rangle})^{-1}\cU \cB_{R-\delta}.
    \end{equation}

Let $L$ be the maximum length of the elements in $\Lambda^\jmath_{\imath}(\ua)$ over all $\ua\in \cA'$, $\imath\in \ldbrack 1,n\rdbrack$, and $\jmath \in \ldbrack 1, m\rdbrack$. 

We now proceed to build the required sequences inductively. First, pick an arbitrary element $\ua_0 \in \cA'$ and let $x_{0,0}$ be any point in $\cylinder[-k;\ua_0]\cap \Sigma_H$. We also set $n_{0,0}=0$.

The following claim establishes the induction step. Note that we write $\vee$ instead of $\vee_{2k+1}$ for simplicity. 

\medskip

    \noindent
    \textbf{Claim.} 
    There exists $\varepsilon>0$ and a family of words $\ua_{i,j}$ for $i \geq 0$ and $j \in \ldbrack 0, m^i - 1 \rdbrack$ such that for every $\alpha$-H\"older cocycle $(\sigma,\tilde{A})$ with $d_{\cC^\alpha}(A,\tilde{A})<\varepsilon$, and for all $i\in \N$ and $j\in \ldbrack 0, m^i-1\rdbrack$, 
    $$ \tilde{A}_{n_{i,j}}(y)\in \cU\cB_{R} \quad \text{for every } y\in \cylinder[-k;\uw_{1,j_1}\vee\cdots \vee \uw_{i,j}]\cap W^u_{loc}(x_0), $$
    where $j_i=j$, $j_{{l}}= \lfloor \frac{j_{{l}+1}}{m} \rfloor$ for ${l}<i$, $\uw_{{l},j_{{l}}} \in T(\ua_{{l}-1,j_{{l}-1}},\ua_{{l},j_{{l}}})$, and $n_{i,j}:=-k+\sum_{{l}\leq i}|\uw_{{l},j_{{l}}}|$. 
    Moreover, for every $i \geq 0$ and every $0 \leq j < j' < m^i$, 
    $$ \cylinder[-k;\uw_{1,j_1}\vee\cdots \vee \uw_{i,j}]\cap \cylinder[-k;\uw_{1,j'_1}\vee\cdots \vee \uw_{i,j'}]=\emptyset.$$

   The proof of the claim proceeds by induction on $i$. The base case $i=0$ is trivial. Now, assume the claim holds for some integer $i \geq 0$, then we will establish it for $i+1$. Let $j \in \ldbrack 0, m^i-1\rdbrack$ be an index from the $i$-th step, associated with the letter $\ua_{i,j} \in \cA'$. By the inductive hypothesis, for $y$ in the specified cylinder, $\tilde{A}_{n_{i,j}}(y) \in \cU\cB_{R}$. Since $\cU = \bigcup_{\imath=1}^n \cU_\imath$, there exists some $\imath \in \ldbrack 1, n\rdbrack$ such that $\tilde{A}_{n_{i,j}}(y) \in \cU_\imath$.
By equation \eqref{eq:smaller-u-t}, for this $\ua_{i,j}$ and $\imath$, we have $m$ families of transitions $\Lambda^\jmath_\imath(\ua_{i,j})$ indexed by $\jmath \in \ldbrack 1, m\rdbrack$. 
We define the $m$ indices associated with $j$ at step $i+1$ by $j' = m j + \jmath - 1$ (for $\jmath \in \ldbrack 1, m\rdbrack$). For each such $j'$, we select a word $\uw_{i+1, j'} \in \Lambda^\jmath_\imath(\ua_{i,j})$ and denote its terminal symbol in $\cA'$ by $\ua_{i+1, j'}$.

By choosing $\varepsilon > 0$ sufficiently small, we guarantee that for any $\alpha$-H\"older cocycle $\tilde{A}$ with $d_{\cC^\alpha}(A,\tilde{A}) < \varepsilon$, every fiberwise product of $\tilde{A}$ of length at most $L$ is uniformly close to the corresponding product for $A$. Specifically, we can ensure that the perturbation over the transition word $\uw_{i+1, j'}$ is bounded by $\delta$. Combined with the inclusion \eqref{eq:smaller-u-t}, this guarantees that for the perturbed cocycle $\tilde{A}$, we have $\tilde{A}_{n_{i+1,j'}}(y) \in \cU\cB_R$ as well. The analysis of this step is completely analogous to the arguments presented in the proof of \cite[Proposition 5.6]{NRR1}. 

Finally, the disjointness of the cylinders at step $i+1$ follows immediately from the construction, as for a fixed branch $j$, the $m$ distinct transition sets $\Lambda^\jmath_\imath(\ua_{i,j})$ generate mutually disjoint extensions within the subshift of finite type.

This completes the proof by setting $\cV:=\cU \cB_R$. Furthermore, to verify property (iii), for any given valid seque

nce of branches leading to $j$, we define the sequence of intermediate times by $t_l := n_{l,j_l}$ for $l=1, \dots, i$, where $j_l$ represents the ancestor index at step $l$.
\end{proof}

\begin{figure}
    \centering
    \includegraphics[width=\linewidth]{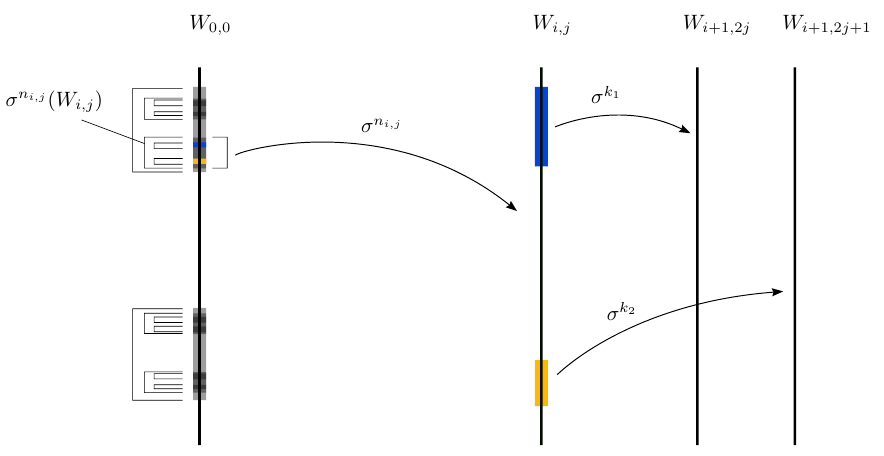}
    \caption{Induction process to define $W_{i+1, j'}$ from $W_{i, j}$. 
    Illustration of the family of nested intervals generating a Cantor set of points in $W_{0,0}$ with bounded fiberwise orbits.
    In this figure, $m=2$, and $n_{i+1, 2j}:=n_{i, j}+k_1$, $n_{i+1, 2j+1}:=n_{i, j}+k_2$.}
    \label{fig:nested-intervals}
\end{figure}

\begin{proof}[Proof of Theorem \ref{thm:C}]
    Suppose that $(\sigma,A)$ satisfies the hypotheses of the theorem. By Proposition \ref{prop:stability-H\"older-calculation}, there exists $\varepsilon > 0$ such that for every $\alpha$-H\"older cocycle $(\sigma, \tilde{A})$ with $d_{\cC^{\alpha}}(A, \tilde{A}) < \varepsilon$, we find:
    \begin{itemize}
        \item a positive integer $L$,
        \item a precompact non-empty open set $\cV\subset G$,
        \item sequences of integers $\{n_{i,j}\}$ and points $\{x_{i,j}\}$ in $\Sigma_H$, both indexed by $i \geq 0$ and $j \in \ldbrack 0, m^i-1 \rdbrack$. 
    \end{itemize}
    Setting $W_{i,j}:=W^u_{loc}\big(\sigma^{n_{i,j}}({x_{i,j}})\big)$ satisfies the properties in the conclusion of Proposition \ref{prop:stability-H\"older-calculation}. 

    Let $\Sigma' = \{0, \dots, m-1\}^\N$ be the full shift on $m$ symbols. For any infinite sequence $s = (s_1, s_2, \dots) \in \Sigma'$, we inductively define a sequence of indices $\{j_i(s)\}_{i\geq 0}$ by setting $j_0(s) = 0$ and $j_i(s) = m j_{i-1}(s) + s_i$ for $i \geq 1$. By property (i) of Proposition \ref{prop:stability-H\"older-calculation}, the sets $\sigma^{-n_{i,j_i(s)}}(W_{i,j_i(s)})$ are compact and nested, meaning their intersection is non-empty. Since the diameters of these preimages shrink to zero, the intersection defines a unique point. We define the map $\Phi: \Sigma' \to \Sigma_H$ by
    $$ \Phi(s) := \bigcap_{i \geq 0} \sigma^{-n_{i,j_i(s)}}(W_{i,j_i(s)}). $$
    By the disjointness of the sets at each level $i$ (property (ii)), $\Phi$ is injective. Moreover, sequences in $\Sigma'$ that agree up to the $i$-th position are mapped into the same compact set $\sigma^{-n_{i,j_i(s)}}(W_{i,j_i(s)})$. Thus, $\Phi$ is continuous, making it a topological embedding of $\Sigma'$ into $\Sigma_H$. We denote the image of this  embedding 
    \[X := \Phi(\Sigma').\]

    We now show that for any $x \in X$, the sequence $\{\|\tilde{A}_n(x)^{\pm 1}\|\}_{n\in \N}$ is bounded. By property (iii) of Proposition \ref{prop:stability-H\"older-calculation}, there exists an infinite sequence of times $0 = t_0 < t_1 < t_2 < \cdots$ satisfying $t_l - t_{l-1} \leq L$ such that $\tilde{A}_{t_l}(x) \in \cV$ for all $l \geq 1$. Therefore,
    $$ \|\tilde{A}_{t_l}(x)^{\pm 1}\| \leq \kappa_{\overline{\cV}} := \sup\{\|v^{\pm 1}\| \colon v \in \overline{\cV}\}. $$

    Now take any $n \in \N$. There exists some $l \geq 0$ such that $t_l \leq n < t_{l+1}$. Using the cocycle property $\tilde{A}_n(x)=\tilde{A}_{n-t_l}(\sigma^{t_l}(x))\tilde{A}_{t_l}(x)$, we can bound $\|\tilde{A}_n(x)\|$. Since $n - t_l < L$, the norm of the first factor is bounded uniformly by noting that $d_{\cC^{\alpha}}(A, \tilde{A}) < \varepsilon$. Indeed, letting $M_A := \sup_{y \in \Sigma_H} \|A(y)^{\pm 1}\|$, we have 
    \[\|\tilde{A}_{n-t_l}(\sigma^{t_l}(x))^{\pm 1}\|\leq (M_A+\varepsilon)^{n-t_l}\leq M^L,\]
    where $M := \max\{M_A + \varepsilon,1\}$. Importantly, $M$ depends only on $A$, $\varepsilon$, and $L$, and not on the specific choice of $\tilde{A}$. Therefore, we obtain a uniform bound $\kappa$ for the forward orbit:
        $$ \|\tilde{A}_n(x)^{\pm 1}\| \leq M^L \cdot \|\tilde{A}_{t_l}(x)^{\pm 1}\| \leq M^L \kappa_{\overline{\cV}}. $$
    
    It follows from Lemma \ref{lem:b-properties} parts (ii) and (iii) that the omega-limit set of $X$,
    $$ \omega(X) = \bigcap_{N \geq 0} \overline{\bigcup_{n \geq N} \sigma^n(X)}, $$
    is a closed subset of ${\bfB}(\tilde{A},\kappa)$ for $\kappa:=(M^L \kappa_{\overline{\cV}})^2$, that is invariant under both $\sigma$ and $\sigma^{-1}$. Note that $\omega(X)$ is non-empty as it is the intersection of a nested sequence of non-empty compact sets. Consequently, ${\bfB}(\tilde{A},\kappa)$ is also non-empty. 
    
    Next, we proceed by giving a lower bound for the topological entropy of $\sigma\big|_{\omega(X)}$.  
    Let $\cL_N(\omega(X))$ denote the set of all admissible words of length $N$ that appear in the sequences of $\omega(X)$. A word of length $N$ belongs to $\cL_N(\omega(X))$ if and only if it appears as a subword in $\sigma^n(X)$ for arbitrarily large $n$.
    
    Recall from our inductive construction that the sequences of branching times $\{n_{i,j}\}$ strictly increase to infinity, with gaps $n_{i+1, j'} - n_{i,j}$ bounded by $L$. Consider any time index $T > 0$ arbitrarily large. For any block of length $N$, a point $x \in \sigma^T(X)$ has undergone at least $\lfloor N/L \rfloor$ branching events within that window, because the gap between branching times $n_{i+1, j'} - n_{i, j}$ where $mj\leq j'<m(j+1)$ is strictly bounded by $L$. Hence, the set of distinct words of length $N$ appearing in $\sigma^T(X)$, denoted $\mathcal{L}_N(\sigma^T(X))$, has cardinality at least $m^{\lfloor N/L \rfloor}$. Since the alphabet of the underlying shift space is finite, the total number of possible words of length $N$ is finite, and thus there are only finitely many subsets of words of length $N$. By the pigeonhole principle, there must be a specific set of words $S$, with $|S| \ge m^{\lfloor N/L \rfloor}$, that appears in $\mathcal{L}_N(\sigma^T(X))$ for infinitely many $T$. Thus, $S \subset \mathcal{L}_N(\omega(X))$, which implies:
    $$ |\mathcal{L}_N(\omega(X))| \ge m^{\lfloor N/L \rfloor}. $$

    Now, in view of \eqref{eq:entropy-growth-word}:
    $$ h_{\mathrm{top}}(\sigma\big|_{\omega(X)}) = \lim_{N \to \infty} \frac{1}{N} \log   |\mathcal{L}_N(\omega(X))| \ge \lim_{N \to \infty} \frac{1}{N} \log   \left( m^{\lfloor N/L \rfloor} \right) = \frac{1}{L}\log   m =: \gamma. $$
    Since $m \ge 2$ and $L > 0$, we have $\gamma > 0$. Finally, because $\omega(X) \subset {\bfB}(\tilde{A},\kappa)$ we conclude
    $ h_{\mathrm{top}}(\sigma\big|_{{\bfB}(\tilde{A},\kappa)}) \ge h_{\mathrm{top}}(\sigma\big|_{\omega(X)}) \ge \gamma > 0, $
    which completes the proof.
\end{proof}

\begin{remark}\label{rmk:Hausdorff-dim}Due to the well-known relationship between Hausdorff dimension and topological entropy, our results regarding topological entropy have natural analogues in terms of Hausdorff dimension. Indeed, for any closed subset $X \subset \Sigma_\cA$ that is both forward and backward invariant under $\sigma$, we have
\[\dim_H(X)=\frac{h_{\mathrm{top}}(\sigma|_X)}{\log 2}.\]
The factor of $\log 2$ in the denominator arises from the base $2$ we used to define the metric on $\Sigma_\cA$ in \eqref{eq:dist-shift-space}. The full invariance of $X$ is crucial and cannot be reduced to the assumption of forward invariance alone. For instance, the stable set of a periodic point is merely forward-invariant and possesses a positive Hausdorff dimension, yet the topological entropy of its restricted dynamics is zero. In this context, we emphasize that our results guarantee the existence of a fully invariant set ${\bfB}(A,\kappa)$ exhibiting positive entropy. Furthermore, for locally constant cocycles, the stable set of any point with a fiberwise bounded forward orbit consists entirely of points with the same property; consequently, the positivity of the Hausdorff dimension of $\bigcup_{\kappa>0} {\bfB}(A,\kappa)$ in this setting is guaranteed by the results of \cite{NRR1}. We refer the reader to \cite{PesinClimanhaga,PesinDimesionTheory,LedrappierYoung2} for broader discussions on Hausdorff dimension, its connection to entropy, and related concepts.
\end{remark}

\section{Rigidity}\label{sec:rigidity}
Our goal in this section is to provide the proof of Theorem \ref{thm:upperbound} and to show that the abundance of the set of fiberwise bounded orbits for a linear cocycle over an SFT with positive entropy imposes strict rigidity constraints on the cocycle. More precisely, it forces the cocycle to be isometric, meaning it must preserve a continuous family of inner products on the fibers.

We say the $\gldr$ cocycle $(f,A)$ over a homeomorphism $f:X\to X$ is continuously conjugated to an isometric cocycle if there exists a continuous map $P:X \to \gldr$ such that for every $x\in X$, 
\begin{equation}\label{eq:cocycle-conj}
    P(f(x))^{-1} A(x) P(x) \in \mathrm{O}(d). 
\end{equation}
Recall that  $\mathrm{O}(d)$ is the group of $d\times d$ real orthogonal matrices. Equivalently, this means there exists a continuous family of inner products $\langle \cdot, \cdot \rangle_x$ on $\R^d$ preserved by the cocycle, satisfying $\langle A(x)u, A(x)v \rangle_{f(x)} = \langle u, v \rangle_x$ for all $x \in X$ and $u,v \in \R^d$.

Let $\Sigma_H$ be a non-trivial transitive subshift of finite type and $w \in \cA_H^{\mathrm{fin}}$. We denote $\Sigma_H(w) := \Sigma_H \setminus \bigcup_{k \in \Z} \sigma^{k}(\cylinder [w])$.

\begin{theorem}[{\cite[Theorem 3]{Lind-Perturbations-SFT}}]\label{Lem:entropy-gap}
    Let $\Sigma_H$ be a transitive SFT of positive entropy with Perron eigenvalue $\rho(H)>1$. There exist constants $c_H,d_H>0$ and $N_H$ such that for every $w \in \cA_H^{n}$ with $n \geq N_H$,
    \[
c_H\,\rho(H)^{-{n}}
\;\le\;
h_{\mathrm{top}}(\sigma\big|_{\Sigma_H}) - h_{\mathrm{top}}\bigl(\sigma\big|_{\Sigma_H(w)}\bigr)
\;\le\;
d_H\,\rho(H)^{-{n}}.\]
\end{theorem}

\begin{proof}[Proof of Theorem \ref{thm:upperbound}]
First, note that for every transitive subshift of finite type $\sigma:\Sigma_H\to \Sigma_H$, there exists $k\in \N$ and a partition of $\Sigma_H$ into $k$ disjoint closed subsets $\Sigma_1,\ldots,\Sigma_k$ such that for every $i$, $\sigma(\Sigma_i)=\Sigma_{i+1}$ (taking indices mod $k$), and for every $i$, $\sigma^k\big|_{\Sigma_i}$ is a mixing SFT. 

Let $x\in \Sigma_H$ be a point with a dense forward orbit. If $\sup_{n\in \N} \|A_n(x)^{\pm 1}\|<+\infty$, then by Lemma \ref{lem:b-properties} the set $\{A_n(y):y\in \Sigma_H, n\in \Z\}$ has compact closure in $\gldr$. Consequently, by \cite[Theorem 1.3]{KalininSadovskaya2018}, the cocycle $(\sigma^k|_{\Sigma_1},A_k)$ preserves a continuous family of norms on fibers, which implies that the cocycle $(\sigma,A)$ preserves a continuous family of norms. 

Therefore, since the cocycle $(\sigma,A)$ satisfies the assumptions of the theorem but is not continuously conjugate to an isometric cocycle, it follows that for our chosen point $x\in\Sigma_H$ with a dense forward orbit, either the sequence $\{\|A_n(x)\|\}_{n\in \N}$ or $ \{\|A_n(x)^{-1}\|)\}_{n\in \N}$ must be unbounded. Choose $N\in\N$ such that \[\max(\|A_N(x)\|, \|A_N(x)^{-1}\|)>\kappa^2.\] Since for a fixed $N$ the maps $(y, \tilde{A}) \mapsto \|\tilde{A}_N(y)\|$ and $(y, \tilde{A}) \mapsto \|\tilde{A}_N(y)^{-1}\|$ are continuous with respect to both the point $y\in \Sigma_H$ and the cocycle in the $\cC^0$ topology, there exists an open neighborhood $U$ of $x$ and a $\cC^0$ neighborhood $\cU$ of $A$ such that for every $y\in U$ and every $\tilde{A}\in \cU$, $\max(\|\tilde{A}_N(y)\|, \|\tilde{A}_N(y)^{-1}\|)>\kappa^2$. 

Since cylinder sets form a basis for the topology of $\Sigma_H$, we may choose a word $w$ of sufficiently large length $L$ for which the cylinder $\cylinder [w]$ is contained in $U$.

Hence according to \eqref{eq:inclusion-B}, the entire cylinder $\cylinder [w]$ is disjoint from ${\bfB}(\tilde{A},\kappa)$ for every $\tilde{A}\in \cU$, and so we have ${\bfB}(\tilde{A},\kappa)\subset \Sigma_H(w)$ for every $\tilde{A} \in \cU$. Therefore, by Lemma~\ref{Lem:entropy-gap}, for some constant $c_H>0$,
$$
h_{\mathrm{top}} (\sigma\big|_{{\bfB}(\tilde{A},\kappa)})\leq h_{\mathrm{top}}\bigl(\sigma\big|_{\Sigma_H(w)}\bigr)
\;\le\;
h_{\mathrm{top}}(\sigma\big|_{\Sigma_H}) -c_H\,\rho(H)^{-L}.
$$
Consequently, for all $\tilde{A}$ sufficiently close to $A$ in the $\cC^0$ topology, the entropy of $\sigma\big|_{{\bfB}(\tilde{A},\kappa)}$ has a uniform gap from $h_{\mathrm{top}}(\sigma\big|_{\Sigma_H})$.
\end{proof}

We conclude the section by an example highlighting the necessity of the H\"older regularity in Theorem \ref{thm:upperbound}. Indeed, the rigidity provided by \emph{this theorem does not hold for $\mathcal{C}^0$ cocycles.}

\medskip
\begin{example}\label{rmk:nonrigid-full-entropy}
Let $\sigma:\Sigma_H\to \Sigma_H$ be a subshift of finite type with positive topological entropy. Using strong connectivity of the associated graph (Lemma \ref{lem:connected}), one can find two distinct finite words $\uv, \uw \in \mathcal{A}_H^{\mathrm{fin}}$ such that any concatenation of them is admissible. Denote the length of $\uv$ by $\ell$ and consider a sequence $\{a_n\}_{n\in \mathbb{N}}$ of real numbers which is bounded (say, $|a_n| \leq 1$), non-convergent, but such that $a_{n+1}-a_n$ converges to zero. Define the function $\phi: \Sigma_H \to \mathbb{R}$ as follows. For every $k \in \mathbb{N}$, if $x \in \cylinder[0;{\uv^k\uw}]\setminus \cylinder[0;{\uv^{k+1}\uw}]$, define $\phi(x) = a_k$ and otherwise, set $\phi(x) = 0$. Next, define the cocycle $A: \Sigma_H \to \mathrm{SL}(2, \mathbb{R})$ by 
$$A(x) = \begin{pmatrix} 1 & \phi(\sigma^\ell x) - \phi(x) \\ 0 & 1 \end{pmatrix}.$$\\
The map $A$ is continuous because it is locally constant for all points except those in the local stable set of the periodic point $p=(\ldots,\uv;\uv,\uv,\uv,\ldots)$.
Moreover, since $\lim_{n\to\infty} (a_{n+1} - a_n) = 0$, we deduce that the cocycle is continuous at points of $W^s_{loc}(p)$ as well. 

On the other hand, for any $k \in \mathbb{N}$ and $x \in \Sigma_H$, the cocycle over $k\ell$ iterates expands as a partially telescoping product:
\begin{equation}\label{eq:A_lk}
    A_{k\ell}(x) = A(\sigma^{k\ell-1} x) \cdots A(x) =  \begin{pmatrix} 1 & \sum_{j=0}^{\ell-1} \left( \phi(\sigma^{k\ell+j} x) - \phi(\sigma^j x) \right) \\ 0 & 1 \end{pmatrix}.
\end{equation}
Since $\phi$ is bounded by $\sup_n |a_n| \leq 1$, the absolute value of the upper-right entry is uniformly bounded by $2\ell$ for all $k$ and $x$. Extending this to $n = k\ell + r$ adds at most $\ell$ bounded terms, so the cocycle over $\sigma$ has uniformly bounded products everywhere.

However, the cocycle is not conjugate to an isometric cocycle. Suppose, for a contradiction, that there exists a continuous map $P:\Sigma_H\to \GL(2,\mathbb{R})$ such that
$$O(x)=P(\sigma x)A(x)P(x)^{-1}\in \mathrm{O}(2) \qquad \text{for every }x\in \Sigma_H.$$ 
Now, for each $k\in \mathbb{N}$ define 
$$x^{(k)}=(\ldots,\uv;\underbrace{\uv\ldots\uv}_k,\uw,\uw,\ldots).$$
Clearly, $x^{(k)}\to p$. Let us evaluate the cocycle \eqref{eq:A_lk} at $x^{(k)}$. For $1 \leq j \leq \ell-1$, the shifted sequence $\sigma^j x^{(k)}$ does not begin with the word $\uv$, and therefore $\phi(\sigma^j x^{(k)}) = 0$. Similarly, since $\sigma^{k\ell} x^{(k)}$ begins with the word $\uw$, we have $\phi(\sigma^{k\ell+j} x^{(k)}) = 0$ for all $0 \leq j \leq \ell-1$. The only non-vanishing term in the sum is at $j=0$, yielding $\phi(x^{(k)}) = a_k$. Thus, 
\[A_{k\ell}(x^{(k)})= \begin{pmatrix} 1 & -a_k\\ 0 & 1 \end{pmatrix}.\]
Now choose two subsequences $k_j,k'_j\to\infty$ such that
$a_{k_j}\to \beta $ and $ a_{k'_j}\to \beta'$
with $\beta\neq \beta'$. This is possible since $\{a_k\}$ is bounded and non-convergent. Since $x^{(k)}\to p$, by continuity of $P$ we obtain
$$P(x^{(k_j)})\to P(p), \qquad P(x^{(k'_j)})\to P(p).$$
Let $q = (\ldots,\uv;\uw,\uw,\uw,\ldots)$ and note that $\sigma^{k\ell}( x^{(k)}) = q$ for all $k \in \mathbb{N}$. Therefore,
$$O_{k_j\ell}(x^{(k_j)})= P(q) \begin{pmatrix} 1 & -a_{k_j}\\ 0 & 1 \end{pmatrix} P(x^{(k_j)})^{-1} \longrightarrow S:= P(q) \begin{pmatrix} 1 & -\beta\\ 0 & 1 \end{pmatrix} P(p)^{-1},$$
and similarly
$$O_{k'_j\ell}(x^{(k'_j)}) \longrightarrow S':= P(q) \begin{pmatrix} 1 & -\beta'\\ 0 & 1 \end{pmatrix} P(p)^{-1}.$$
Since  $O_n(\cdot)$ takes values in $\mathrm{O}(2)$, and $\mathrm{O}(2)$ is compact, both $S,S'$ belong to $\mathrm{O}(2)$. Hence the product $S^{-1}S'$ also belongs to $\mathrm{O}(2)$:
\[
S^{-1}S'
=
P(p)
\begin{pmatrix}
1 & \beta-\beta'\\
0 & 1
\end{pmatrix}
P(p)^{-1}
\in \mathrm{O}(2).
\]
This is impossible because $S^{-1}S'$ is conjugate to a matrix with trace $2$ and therefore it must have trace $2$ itself. The only element in $\mathrm{O}(2)$ with trace $2$ is the identity matrix $\id$. This forces $S=S'$, which implies $\beta = \beta'$, contradicting the assumption that $\beta \neq \beta'$. \hfill \qedsymbol
\end{example}

\section{Proof of Theorem \ref{thm:A}}
\label{sec:proof-main-thm}
This section is devoted to the proof of Theorem \ref{thm:A} using the tools established in previous sections. The proof uses the fact that any continuous cocycle over an SFT with positive topological entropy without any dominated splitting admits arbitrarily small $\cC^0$-perturbations satisfying a covering condition which, due to the multiple covering principle, ensures multiple covering for some suitably chosen sets. The core idea of the possibility of such perturbations can be seen in the case of cocycles over full shifts presented in \cite[Section 7]{NRR1}. However, modifying the proof to accommodate general SFTs required developing additional notation and was omitted there. For the sake of completeness, we provide a rigorous proof of this fact here.

\begin{lemma}\label{lem:distinct-transitions}
Let $\sigma \colon \Sigma_H \to \Sigma_H$ be a transitive subshift of finite type with positive topological entropy. For any two admissible words $\underline{p}, \underline{p}'$ and any integers $t, L \ge 1$, there exists an integer $L' > L$ such that the set $T(\underline{p},\underline{p}')$ contains at least $t$ distinct words of length $L'$.
\end{lemma}

\begin{proof}
Consider the directed graph associated with the SFT with vertex set $\cA$. By Lemma \ref{lem:connected}, this graph is strongly connected and has at least one vertex with an out-degree of at least two.

Let $q \in \cA$ be such a vertex, and let $q_1, q_2 \in \cA$ be distinct elements corresponding to two distinct edges originating from $q$. Since the graph is strongly connected, we can find two distinct cycles starting and ending at $q$. For $i=1,2$, let $\underline{c}_i \in T(q,q)$ be the word for such a cycle, chosen so that its second letter is $q_i$. Let the length of $\underline{c}_i$ be $\ell_i+1$. We then define two new words by repetition:
\[\underline{\gamma}_1=\underbrace{\underline{c}_1\vee_1\cdots \vee_1 \underline{c}_1}_{\ell_2 ~\mathrm{times}}, \quad \underline{\gamma}_2=\underbrace{\underline{c}_2\vee_1\cdots \vee_1 \underline{c}_2}_{\ell_1 ~\mathrm{times}}.\]
The resulting words $\underline{\gamma}_1$ and $\underline{\gamma}_2$ both belong to $T(q,q)$ and have the same length $\ell_1\ell_2+1$. They are distinct since their second letters are $q_1$ and $q_2$, respectively.

Again by Lemma \ref{lem:connected}, strong connectivity ensures that for the given admissible words $\underline{p}$ and $\underline{p}'$, there exist transition words ${\uw} \in T(\underline{p}, q)$ and ${\uw'} \in T(q, \underline{p}')$.

Finally, for the given integers $t, L \geq 1$, we must construct at least $t$ words of a length $L' > L$. Now, choose an integer $r$ large enough to satisfy  $2^r \geq t$ and 
     $|\uw| + r\ell_1\ell_2 + |\uw'|-1 > L$. 
 We then form $2^r$ distinct words in $T(\underline{p},\underline{p}')$ for any sequence $(i_1, \dots, i_r) \in \{1, 2\}^r$ of the form:
\[ \uw\vee_1 \underline{\gamma}_{i_1} \vee_1 \underline{\gamma}_{i_2} \vee_1 \cdots \vee_1 \underline{\gamma}_{i_r}\vee_1 \uw'.\]
All $2^r$ of these words have the same length, $L'>L$. This completes the proof.
\end{proof}

\begin{proposition}\label{prop:density-covering}

Let $d\in \N$ and $G$ be a closed subgroup of $\gldr$. 
Let $(\sigma, A)$ be a locally constant $G$ cocycle over the transitive subshift of finite type $\sigma \colon \Sigma_H \to \Sigma_H$ with $h_{\mathrm{top}}(\sigma\big|_{\Sigma_H})>0$.
If $A_{n}(p) = \id$ for some periodic point $p$ of period $n$, then for every $\varepsilon > 0$, there exists a locally constant cocycle $(\sigma, \tilde{A})$ with $d_{\cC^0}(A, \tilde{A}) < \varepsilon$ that satisfies the covering condition \eqref{eq:gen-covering}.
\end{proposition}
\begin{proof}
This proof is an adaptation of the arguments in Propositions 7.3 and 7.6 of \cite{NRR1}. Assume that the cocycle $(\sigma,A)$ depends on positions in $\ldbrack -k,+k\rdbrack$. 

Let $\cU$ be an arbitrary precompact, non-empty, open subset of $G$ containing the identity element. Since $\overline{\cU}\subset \overline{\cU}\cU$ (see Remark \ref{rmk:covering-improvement}), the family of open sets $\{u\cU\}_{u\in \overline{\cU}}$ forms an open cover of $\overline{\cU}$. By compactness, there is a finite subcovering. Therefore, we may find a finite subset $\cG=\{g_1, \dots, g_t\}$ of $G$ such that $\overline{\cU}\subset \cG^{-1}\cU$. 
Moreover, because $\sigma\colon \Sigma_H\to \Sigma_H$ is transitive and has positive topological entropy, Lemma \ref{lem:distinct-transitions} (with $\underline{p}=\underline{p}'$) guarantees the existence of $t$ distinct words $\underline{w}_1, \dots, \underline{w}_t$ of the same length such that the words $\underline{p} \uw_i$ and $\uw_i\underline{p}$ for $i\in \ldbrack 1, t\rdbrack$ are all admissible. 
   
   Now, let $\ell_1,\ell_2$ be positive integers such that the admissible words $\underline{\tilde{p}}:=\underline{p}^{\ell_1}$  and $\tilde{\uw}_i:=(\underline{w}_i \underline{p})^{\ell_2}$ have the same length, denoted by $\ell$. We need also ensure that these words are distinct for the rest of the proof. Note that the words $\tilde{\uw}_i$ are distinct from each other because the $\uw_i$ are. If $\tilde{\uw}_i=\underline{\tilde{p}}$ for some $i$, then $\underline{w}_i$ must be a subword of $\underline{p}^{\ell_1}$, and denoting the period of $p$ by $n$, there are at most $n$ such distinct subwords. Thus, if we initially apply Lemma \ref{lem:distinct-transitions} to find $n+t$ words, we guarantee that at least $t$ of the resulting words $\tilde{\underline{w}}_i$ are distinct from $\tilde{\underline{p}}$.
  
  Now define the words $\underline{v}_i := \underline{\tilde{p}}^N\underline{\tilde{w}}_i\underline{\tilde{p}}^N$, which are all admissible, where $N$ is a large positive integer that will be specified later. From the hypothesis $A_{n}(p)=\id$, it follows that the cocycle product over the word $\underline{\tilde{p}}=\underline{p}^{\ell_1}$ is also the identity. Consequently, for every $z\in \cylinder[-k;\underline{v}_i]$, the cocycle products over the initial and final blocks of $\underline{\tilde{p}}$ are identity matrices:
   $$A_{\ell}(z)=A_{\ell}(\sigma^\ell (z))=\cdots =A_{\ell}({\sigma^{(N-1)\ell}}(z))=\id.$$
   
   Now, provided that $N$ is chosen to be sufficiently large depending on $\varepsilon$ and the norms of $g_i$'s, we can proceed similarly to the proof of Propositions 7.3 and 7.6 in \cite{NRR1}. Indeed, we can perturb the cocycle $A$ on the cylinders of the form 
   \[\cylinder[-k;\underline{\tilde{p}}^j \underline{\tilde{w}}_i \underline{\tilde{p}}^{N-j}], \quad \text{for}~i\in \ldbrack 1, t\rdbrack, ~ j\in \ldbrack 0, N-1\rdbrack.\]
   Note that these cylinders are mutually disjoint, as the words $\tilde{w}_i$ and $\tilde{p}$ are distinct and have the same length. We can therefore perturb the cocycle on these cylinders (in the $\cC^0$ topology) to obtain $A_{[\underline{v}_i; {(N+1)\ell}\rangle}=g_i$. Since $\underline{v}_i \in T(\underline{\tilde{p}},\underline{\tilde{p}})$, this gives the covering condition. 
\end{proof}

Now, we have all the required ingredients for the proof of Theorem \ref{thm:A}. 
\begin{theorem}[Theorem A]\label{thm:A'}
Let $d\in \N$ and $\alpha>0$. Suppose $\sigma:\Sigma_H\to \Sigma_H$ is a transitive subshift of finite type with positive topological entropy, and let $(\sigma,A)$ be a continuous $\sldr$ cocycle. If $(\sigma, A)$ admits no dominated splitting, then for every $\varepsilon>0$, there exists a $\cC^\alpha$-open set $\cB$ of $\sldr$ cocycles over $\sigma$ and positive constants $\kappa,\gamma,\gamma'$ with $\gamma\leq \gamma'<h_{\rm top}(\sigma\big|_{\Sigma_H})$ such that every $(\sigma,\tilde{A}) \in \cB$ satisfies $d_{\cC^0}(A,\tilde{A})<\varepsilon$ and
\[\gamma \leq h_{\mathrm{top}}(\sigma\big|_{{\bfB}(\tilde{A},\kappa)})\leq \gamma'.\]
\end{theorem}

\begin{proof} By utilizing  Proposition \ref{prop:density-covering}, one can closely follow the proof of \cite[Theorem~7.1]{NRR1}.

If the cocycle $(\sigma,A)$ admits no dominated splitting, then by \cite{BDP} we can perturb it in the $\cC^0$ topology to obtain a periodic point $p\in \Sigma_H$ of some period $n\geq 1$ for $\sigma$ such that $A_n(p)=\id$ (see \cite[Section 6]{NRR1} for more details). Then, using Proposition \ref{prop:density-covering}, after a second sufficiently small $\cC^0$ perturbation, we obtain a locally constant cocycle $(\sigma, A')$  with $d_{\cC^0}(A,A')<\varepsilon$ satisfying the covering condition \eqref{eq:gen-covering}. By Theorem \ref{thm:multi-cov}, this covering condition is equivalent to the multiple covering condition. Consequently, Theorem \ref{thm:C} ensures that there exists an open neighbourhood of $A'$ in the $\cC^\alpha$ topology, and positive constants $\kappa, \gamma$ such that for every $(\sigma,\tilde{A})$ in this open neighborhood, $h_{\rm top}(\sigma\big|_{{\bfB}(\tilde{A},\kappa)})\geq \gamma$.  By shrinking this neighborhood, we can guarantee that all its elements have a $\cC^0$ distance less than $\varepsilon$ from the initial cocycle $(\sigma,A)$. Denote this neighborhood by $\cB'$.

Now, by \cite[Theorem~A]{Viana08}, the set of cocycles which are not isometric is $\cC^\alpha$-dense in $\cB'$. Note that the statement of the mentioned theorem in \cite{Viana08} is for differentiable maps but it holds for hyperbolic homeomorphisms such as SFTs with positive entropy (see Subsection 1.3 in \cite{Viana08}). The invariant measure in the statement of \cite[Theorem~A]{Viana08} should be taken as the measure of maximal entropy for $\sigma\big|_{\Sigma_H}$, which is well-known to exist and possess product structure. 

Now, take an arbitrary cocycle $(\sigma, A'')\in \cB'$ with a positive Lyapunov exponent at some point. In particular, this cocycle is not continuously conjugate to an isometric cocycle. Thus, by Theorem \ref{thm:upperbound}, there is $\gamma'<h_{\rm top}(\sigma\big|_{\Sigma_H})$ and a $\cC^0$-neighborhood $\cB''$ of $(\sigma,A'')$ such that every cocycle in $\cB''$ has $h_{\rm top}(\sigma\big|_{{\bfB}(\tilde{A},\kappa)}) \leq \gamma'$. Now, for the conclusion of the proof of Theorem \ref{thm:A'}, it suffices to take $\cB=\cB'\cap \cB''$, which is an open set in the $\cC^\alpha$ topology (note that $\cC^0$-open sets are open in the $\cC^\alpha$ topology as well), and the proof is complete. 
\end{proof}

 \section{Further comments and questions}\label{sec:questions}

In this section, we propose two directions for future research arising from our results.

 \subsection{Bounded orbits vs. orbits with subexponential growth}
Clearly, for every base point whose orbit yields a uniformly bounded fiberwise product of matrices, all Lyapunov exponents exist and are equal to zero. Therefore, a trivial upper bound for the topological entropy of the set of fiberwise bounded orbits is given by the topological entropy of the set of points whose Lyapunov exponents vanish in the fiber. The following question asks whether these two quantities always coincide.

\begin{question}
Let $(\sigma, A)$ be a continuous $\mathrm{SL}(2,\R)$ cocycle over a subshift of finite type $\sigma : \Sigma_H \to \Sigma_H$. Let $\mathrm{B}$ denote the set of points with fiberwise bounded orbits. Is it true that the topological entropy of $\mathrm{B}$ equals the topological entropy of the set of points whose Lyapunov spectrum degenerates? That is, 
\[
h_{\mathrm{top}}(\sigma\big|_{\mathrm{B}}) = h_{\mathrm{top}}(\sigma\big|_{\Lambda^{\star}}),
\]
where $\Lambda^{\star} = \left\{ x \in \Sigma_H \,:\, \lim_{n \to \infty} \frac{1}{n} \log   \|A_n(x)\| = 0 \right\}$?
\end{question}

This question is already nontrivial and interesting in the case of locally constant cocycles depending only on the zero coordinate. Another important special case arises for one-dimensional cocycles. In this setting, the answer to the above question is known to be positive for locally constant cocycles depending on the zero coordinate. The abelian nature of the problem simplifies the analysis and allows one to reinterpret it in terms of random walks on the real line.

This one-dimensional case is closely related to the literature on the robust existence of invariant measures with zero Lyapunov exponents for diffeomorphisms.

\subsection{Dependence of topological entropy on parameters}

Our results establish the positivity of the topological entropy of the dynamics restricted to the set of fiberwise bounded orbits. However, we only obtain a lower bound and do not provide an explicit formula for this entropy. It would be quite interesting to compute the topological entropy more precisely, at least in the case of a full shift generated by finitely many matrices.

A natural related problem concerns parameter dependence: when the underlying symbolic dynamics is fixed, how does the topological entropy of the fiberwise bounded set vary as a function of the entries of the generating matrices? The following example serves as a first step in this direction.

For $\theta \in \R$, consider the locally constant cocycle over the full shift on two symbols generated by the two matrices 
$\left(\begin{smallmatrix} 2 & 0 \\ 0 &1/2 \end{smallmatrix}\right)$ and 
$\left(\begin{smallmatrix} \cos \theta & -\sin \theta \\ \sin \theta & \cos \theta \end{smallmatrix}\right)$ in $\mathrm{SL}(2,\R)$.
Let $\mathrm{B}_\theta$ denote the set of points with fiberwise bounded orbits in the cocycle associated with these matrices as generators. The following problem concerns the dependence of the topological entropy on the parameter $\theta$.

\begin{question}
Is the function $\theta \mapsto h_{\mathrm{top}}(\sigma\big|_{\mathrm{B}_\theta})$ continuous?  
Is it strictly increasing on some interval $[0, \theta_0]$?
\end{question}

\bibliographystyle{amsalpha}  
\providecommand{\bysame}{\leavevmode\hbox to3em{\hrulefill}\thinspace}
\providecommand{\MR}{\relax\ifhmode\unskip\space\fi MR }
\providecommand{\MRhref}[2]{%
  \href{http://www.ams.org/mathscinet-getitem?mr=#1}{#2}
}
\providecommand{\href}[2]{#2}

\end{document}

%% file: macros.tex
\usepackage{marginnote}

\newcommand{\exwar}[1]{}

\DeclareMathOperator{\GL}{\rm{GL}}

\newcommand{\cylinder}{\mathsf{C}}

\newcommand{\bfB}{{\bf B}} 
\newcommand{\ua}{\underline{a}}
\newcommand{\ub}{\underline{b}}

\newcommand{\uw}{\underline{w}}
\newcommand{\ud}{\underline{d}}

\newcommand{\R}{\mathbb{R}}\newcommand{\N}{\mathbb{N}}
\newcommand{\Z}{\mathbb{Z}}

\newcommand{\sldr}{\mathrm{SL}(d,\mathbb{R})}
\newcommand{\gldr}{\mathrm{GL}(d,\R)}

\newcommand{\id}{\mathrm{Id}}

\newcommand{\cA}{\mathcal{A}}
\newcommand{\cB}{\mathcal{B}}
\newcommand{\cC}{\mathcal{C}}

\newcommand{\cG}{\mathcal{G}}

\newcommand{\cL}{\mathcal{L}}

\newcommand{\cO}{\mathcal{O}}
\newcommand{\cP}{\mathcal{P}}

\newcommand{\cS}{\mathcal{S}}
\newcommand{\cT}{\mathcal{T}}
\newcommand{\cU}{\mathcal{U}}
\newcommand{\cV}{\mathcal{V}}
\newcommand{\cW}{\mathcal{W}}

\newcommand{\cZ}{\mathcal{Z}}

\numberwithin{equation}{section}
\newtheorem{theorem}{Theorem}[section] 

\newtheorem{lemma}[theorem]{Lemma}
\newtheorem*{lemma*}{Lemma}
\newtheorem{proposition}[theorem]{Proposition}
\newtheorem*{proposition*}{Proposition}

\newtheorem{question}[theorem]{Question}

\newtheorem*{question*}{Question}
\newtheorem*{theorem*}{Theorem}
\newtheorem*{claim*}{Claim}

\newtheorem{theoremain}{Theorem}

\theoremstyle{definition}
\newtheorem{definition}[theorem]{Definition}

\newtheorem{example}[theorem]{Example}

\theoremstyle{remark}
\newtheorem{remark}[theorem]{Remark}